\documentclass{article}
\title{Aspects of Predicative Algebraic Set Theory II: Realizability}
\author{Benno van den Berg \& Ieke Moerdijk}
\date{January 15, 2008}
\usepackage{ast2}
\begin{document}
\maketitle

\emph{Dedicated to Jean-Yves Girard on the occasion of his 60th birthday}

\section{Introduction}

This paper is the second in a series on the relation between algebraic set theory \cite{joyalmoerdijk95} and predicative formal systems. The purpose of the present paper is to show how realizability models of constructive set theories fit into the framework of algebraic set theory. It can be read independently from the first part \cite{bergmoerdijk07b}; however, we recommend that readers of this paper read the introduction to \cite{bergmoerdijk07b}, where the general methods and goals of algebraic set theory are explained in more detail.

To motivate our methods, let us recall the construction of Hyland's effective topos $\Eff$ \cite{hyland82}. The objects of this category are pairs $(X, =)$, where = is a subset of $\NN \times X \times X$ satisfying certain conditions. If we write $n \Vdash x = y$ in case the triple $(n, x, y)$ belongs to this subset, then these conditions can be formulated by requiring the existence of natural numbers $s$ and $t$ such that
\begin{displaymath}
\begin{array}{l}
s \Vdash x = x' \rightarrow x' = x \\
t \Vdash x = x' \land x' = x'' \rightarrow x = x''.
\end{array}
\end{displaymath}
These conditions have to be read in the way usual in realizability \cite{troelstra98}. So the first says that for any natural number $n$ satisfying $n \Vdash x = x'$, the expression $s(n)$ should be defined and be such that $s(n) \Vdash x' = x$. \footnote{For any two natural numbers $n, m$, the Kleene application of $n$ to $m$ will be written $n(m)$, even when it is undefined. When it is defined, this will be indicated by $n(m) \downarrow$. We also assume that some recursive pairing operation has been fixed, with the associated projections being recursive. The pairing of two natural numbers $n$ and $m$ will be denoted by $\langle n, m \rangle$. Every natural number $n$ will code a pair, with its first and second projection denoted by $n_0$ and $n_1$, respectively.} And the second stipulates that for any pair of natural numbers $n$ and $m$ with $n \Vdash x = x'$ and $m \Vdash x' = x''$, the expression $t( \langle n, m \rangle)$ is defined and is such that $t(\langle n, m \rangle) \Vdash x = x''$.

The arrows $[F]$ between two such objects $(X,=)$ and $(Y,=)$ are equivalence
classes of subsets $F$ of $\NN \times X \times Y$ satisfying certain conditions. Writing $n \Vdash Fxy$ for $(n,x,y) \in F$, one requires the existence of realizers for statements of the form
\begin{displaymath}
\begin{array}{l}
Fxy \land x = x' \land y = y' \rightarrow Fx'y' \\
Fxy \rightarrow x = x \land y = y \\
Fxy \land Fxy' \rightarrow y = y' \\
x = x \rightarrow \exists y \, Fxy.
\end{array}
\end{displaymath}
Two such subsets $F$ and $G$ represent the same arrow $[F] = [G]$ iff
they are extensionally equal in the sense that
\[ Fxy \leftrightarrow Gxy \]
is realized.

As shown by Hyland, the logical properties of this topos $\Eff$ are quite remarkable. Its first-order arithmetic coincides with the realizability
interpretation of Kleene (1945). The interpretation of the higher types in $\Eff$ is given by {\bf HEO}, the hereditary effective operations. Its higher-order arithmetic is captured by realizability in the manner of Kreisel and Troelstra \cite{troelstra73}, so as to validate the uniformity principle:
\[ \forall X \in {\cal P}\NN \, \exists n \in \NN \, \phi(X, n) \rightarrow \exists n \in \NN \, \forall X \in {\cal P}\NN \, \phi(X, n). \]

The topos $\Eff$ is one in an entire family of ``realizability toposes''
defined over arbitrary partial combinatory algebras (or more general structures modeling computation). The relation between these toposes has been not been completely clarified, although much interesting work has already been done in this direction \cite{pitts81, hyland82, longley95, birkedal00, hofstravanoosten03, hofstra06} (for an overview, see \cite{vanoosten08}). The construction of the topos $\Eff$ and its variants can be internalised in an arbitrary topos. This means in particular that one can construct toposes by iterating (alternating) constructions of sheaf and realizability toposes to obtain interesting models for higher-order intuitionistic arithmetic {\bf HHA}. An example of this phenomenon is the modified realizability topos, which occurs as a closed subtopos of a realizability topos constructed inside a presheaf topos \cite{vanoosten97}.

The purpose of this series of papers is to show that these results are not only valid for toposes as models of {\bf HHA}, but also for certain types of categories equipped with a class of small maps suitable for constructing models of constructive set theories like {\bf IZF} and {\bf CZF}. In the first paper of this series \cite{bergmoerdijk07b}, we have axiomatised this type of categories, and refer to them as ``predicative category with small maps'' (for the convenience of the reader their precise definition is recalled in Appendix B). A basic result from \cite{bergmoerdijk07b} is the following:
\begin{theo}{existmodelsetth}
Every predicative category with small maps $(\ct{E}, \smallmap{S})$ contains a model $(V, \epsilon)$ of set theory. Moreover,
\begin{itemize}
\item[(i)] $(V, \epsilon)$ is a model of {\bf IZF}, whenever the class \smallmap{S} satisfies the axioms {\bf (M)} and {\bf (PS)}.
\item[(ii)] $(V, \epsilon)$ is a model of {\bf CZF}, whenever the class \smallmap{S} satisfies {\bf (F)}.\footnote{The precise formulations of the axioms {\bf (M)}, {\bf (PS)} and {\bf (F)} can be found in Appendix B as well.}
\end{itemize}
\end{theo}
To show that realizability models fit into this picture, we prove that predicative categories with small maps are closed under internal realizability, in the same way that toposes are. More precisely, relative to a given predicative category with small maps $(\ct{E}, \smallmap{S})$, we construct a ``predicative realizability category'' $(\Eff_\ct{E}, \smallmap{S}_\ct{E})$. The main result of this paper will then be:
\begin{theo}{effascatwsmallmaps}
If $(\ct{E}, \smallmap{S})$ is a predicative category with small maps, then so
is $(\Eff_\ct{E}, \smallmap{S}_\ct{E})$. Moreover, if $(\ct{E}, \smallmap{S})$ satisfies {\bf (M), (F)} or {\bf (PS)}, then so does $(\Eff_\ct{E}, \smallmap{S}_\ct{E})$.
\end{theo}
We show this for the pca $\NN$ together with Kleene application, but the result is also valid, when this is replaced by an arbitrary \emph{small} pca \pca{A} in \ct{E}. The proof of the theorem above is technically rather involved, in particular in the case of the additional properties needed to ensure that the model of set theory satisfies the precise axioms of {\bf IZF} and {\bf CZF}. However, once this work is out of the way, one can apply the construction to many different predicative categories with small maps, and show that familiar realizability models of set theory (and some unfamiliar ones) appear in this way.

One of the most basic examples is that where \ct{E} is the category of sets,
and \smallmap{S} is the class of maps between sets whose fibers are all bounded in size by some inaccessible cardinal. The construction underlying \reftheo{effascatwsmallmaps} then produces Hyland's effective topos $\Eff$, together with the class of small maps defined in \cite{joyalmoerdijk95}, which in \cite{kouwenhovenvanoosten05} was shown to lead to the Friedman-McCarty model of {\bf IZF} \cite{friedman73, mccarty84} (we will reprove this in Section 5).

An important point we wish to emphasise is that one can prove all the model's salient properties without constructing it explicitly, using its universal properties instead. We explain this point in more detail. A predicative category with small maps consists of a category \ct{E} and a class of maps \smallmap{S} in it, the intuition being that the objects and morphisms of \ct{E} are classes and class morphisms, and the morphisms in \smallmap{S} are those that have small (i.e., set-sized) fibres. For such predicative categories with small maps, one can prove that the small subobjects functor is representable. This means that there is a \emph{power class object} $\spower(X)$ which classifies the small subobjects of $X$, in the sense that maps $B \rTo \spower(X)$ correspond bijectively to jointly monic diagrams \diag{B & U \ar[r] \ar[l] & X} with $U \rTo B$ small. Under this correspondence, the identity \func{\id}{\spower(X)}{\spower(X)} corresponds to a membership relation \diag{\in_X \ar@{ >->}[r] & X \times \spower X.}

The model of set theory $V$ that every predicative category with small maps contains (\reftheo{existmodelsetth}) is constructed as the initial algebra for the \spower-functor. Set-theoretic membership is interpreted by a subobject $\epsilon \subseteq V \times V$, which one obtains as follows. By Lambek's Lemma, the structure map for this initial algebra $V$ is an isomorphism. We denote it by Int, and its inverse by Ext:
\diag{ \spower V \ar@/^/[rr]^{\rm Int} & & V \ar@/^/[ll]^{\rm Ext}. }
The membership relation \diag{\epsilon \ar@{ >->}[r] & V \times V} is the result of pulling back the usual ``external'' membership relation \diag{\in_V \ar@{ >->}[r] & V \times \spower(V)} along $\id \times {\rm Ext}$.

\reftheo{existmodelsetth} partly owes its applicability to the fact that the theory of the internal model $(V, \epsilon)$ of {\bf IZF} or {\bf CZF} corresponds precisely to what is true in the categorical logic of \ct{E} for the object $V$ and its external membership relation $\in$. This, in turn, corresponds to a large extent to what is true in the categorical logic of \ct{E} for the higher arithmetic types. Indeed, by the isomorphism \func{\rm Ext}{V}{\spower(V)} and its inverse Int, any generalised element \func{a}{X}{V} corresponds to a subobject \diag{{\rm Ext}(a) \ar@{ >->}[r] & X  \times V} with ${\rm Ext}(a) \rTo X$ small, and for two such elements $a$ and $b$, one has that
\begin{itemize}
\item[(i)] $a \in b$ iff $a$ factors through ${\rm Ext}(b)$.
\item[(ii)] $a \subseteq b$ iff the subobject ${\rm Ext}(a)$ of $X \times V$ is contained in ${\rm Ext}(b)$.
\item[(iii)] ${\rm Ext}(\omega) \cong \NN$, the natural numbers object of \ct{E}.
\item[(iv)] ${\rm Ext}(a^b) \cong {\rm Ext}(a)^{{\rm Ext}(b)}$.
\item[(v)] ${\rm Ext}({\cal P}a) \cong \spower({\rm Ext}(a))$.
\end{itemize}
(Properties (i) and (ii) hold by definition; for (iii)-(v), see the proof of Proposition 7.2 in \cite{bergmoerdijk07b}.) Thus, for example, the sentence ``the set of all functions from $\omega$ to $\omega$ is subcountable'' is true in $(V,\epsilon)$ iff the corresponding statement is true for the natural numbers object $\NN$ in the category \ct{E}.

For this reason the realizability model in the effective topos inherits various principles from the ambient category and one immediately concludes:
\begin{coro}{realforIZF}
{\bf IZF} is consistent with the conjunction of the following
axioms: the Axiom of Countable Choice {\bf (AC)}, the Axiom of Relativised Dependent Choice {\bf (RDC)}, the Presentation Axiom {\bf (PA)}, Markov's Principle {\bf (MP)}, Church's Thesis {\bf (CT)}, the Uniformity Principle {\bf (UP)}, Unzerlegbarkeit {\bf (UZ)}, Independce of Premisses for Sets and Numbers {\bf (IP)}, {\bf (IP$_\omega$)}.\footnote{A precise formulation of these principles can be found in Appendix A.}
\end{coro}

Of course, \refcoro{realforIZF} has also been proved directly by realizability \cite{friedman73, mccarty84}; however, it is a basic example which illustrates the general theme, and on which there are many variations. For example, our proof of \reftheo{effascatwsmallmaps} is elementary (in the proof-theoretic sense), hence can be used to prove relative consistency results. If we take for \ct{E} the syntactic category of definable classes in the theory {\bf CZF}, we obtain Rathjen's realizability interpretation of {\bf CZF} \cite{rathjen06}, and deduce:
\begin{coro}{realforCZF} {\rm \cite{rathjen06}}
If {\bf CZF} is consistent, then so is {\bf CZF} combined with the conjunction of the following axioms: the Axiom of Countable Choice {\bf (AC)}, the Axiom of Relativised Dependent Choice {\bf (RDC)}, Markov's Principle {\bf (MP)}, Church's Thesis {\bf (CT)}, the Uniformity Principle {\bf (UP)} and Unzerlegbarkeit {\bf (UZ)}.
\end{coro}
(We also recover the same result for {\bf IZF} within our framework.)

Another possibility is to mix \reftheo{effascatwsmallmaps} with the similar construction for sheaves \cite{bergmoerdijk07a}. This shows that models of set theory ({\bf IZF} or {\bf CZF}) also exist for various other notions of realizability, such as modified realizability in the sense of \cite{vanoosten97, birkedalvanoosten02} or Kleene-Vesley's function realizability \cite{kleenevesley65}. We will discuss this in some more detail in Section 5 below.

Inside Hyland's effective topos, or more generally, in categories of the
form  $\Eff_\ct{E}$ (cf.~\reftheo{effascatwsmallmaps}), other classes of small maps exist, which are not obtained from an earlier class of small maps in \ct{E} by \reftheo{effascatwsmallmaps}, but nonetheless satisfy the conditions sufficient to apply our theorem from \cite{bergmoerdijk07b} yielding models of set theory (cf.~\reftheo{existmodelsetth} above). Following the work of the first author in \cite{berg06}, we will present in some detail one particular
case of this phenomenon,  based on the notion of modest set \cite{hyland88,hylandrobinsonrosolini90}. Already in \cite{joyalmoerdijk95} a class \smallmap{T} inside the effective topos was considered, consisting of those maps which have subcountable fibres (in some suitable sense).  This class does not satisfy the axioms from
\cite{joyalmoerdijk95} necessary to provide a model for {\bf IZF}. However, it was shown in \cite{berg06} that this class \smallmap{T} does satisfy a set of
axioms sufficient to provide a model of the predicative set theory {\bf CZF}.
\begin{theo}{subcountableineff} {\rm \cite{joyalmoerdijk95,berg06}}
The effective topos $\Eff$ and its class of subcountable morphisms \smallmap{T} form a predicative category with small maps. Moreover, \smallmap{T} satisfies the axioms {\bf (M)} and {\bf (F)}.
\end{theo}
We will show that the corresponding model of set theory (\reftheo{existmodelsetth}) fits into the general framework of this
series of papers, and investigate some  of its logical properties, as well
as its relation to some earlier models of Friedman, Streicher and Lubarsky \cite{friedman77,streicher05,lubarsky06}. In particular, we prove:
\begin{coro}{subcimplforCZF}
{\bf CZF} is consistent with the conjunction of the following
axioms: Full separation, the subcountability of all sets, as well as {\bf (AC)}, {\bf (RDC)}, {\bf (PA)}, {\bf (MP)}, {\bf (CT)}, {\bf (UP)}, {\bf (UZ)}, {\bf (IP)} and {\bf (IP$_\omega$)}.
\end{coro}

\noindent
{\bf Acknowledgements:} We would like to thank Thomas Streicher and Jaap van Oosten for comments on an earlier version of this paper, and for making \cite{vanoosten08} available to us.

\section{The category of assemblies}

Recall that our main aim (\reftheo{effascatwsmallmaps}) is to construct for a predicative category with small maps $(\ct{E}, \smallmap{S})$ the realizability category $(\Eff_\ct{E}, \smallmap{S}_\ct{E})$, and show it is again a predicative category with small maps. For this and other purposes, the description of $\Eff$ as an exact (ex/reg) completion of a category of assemblies \cite{carbonifreydscedrov88}, rather than Hyland's original description, is useful. A similar remark applies to the effective topos $\Eff[\pca{A}]$ defined by an arbitrary small pca $\pca{A}$. In \cite{bergmoerdijk07b} we showed that the class of predicative categories with small maps is closed under exact completion. More precisely, we formulated a weaker version of the axioms (a ``category with display maps''; the notion is also recapitulated in Appendix B), and showed that if $(\ct{F},\smallmap{T})$ is a pair satisfying the weaker axioms, then in the exact completion $\overline{\ct{F}}$ of \ct{F}, there is a natural class of arrows $\overline{\smallmap{T}}$, depending on \smallmap{T}, such that the pair $(\overline{\ct{F}}, \overline{\smallmap{T}})$ is a predicative category with small maps. Therefore our strategy in this section will be to construct a category of assemblies relative to the pair $(\ct{E}, \smallmap{S})$ and show it is a category with display maps (strictly speaking, we only need to assume that $(\ct{E}, \smallmap{S})$ is itself a category with display maps for this). Its exact completion will then be considered in the next section.

In this section, $(\ct{E}, \smallmap{S})$ is assumed to be a predicative category with small maps.
\begin{defi}{assembly}
An \emph{assembly} (over \ct{E}) is a pair $(A, \alpha)$ consisting of an object $A$ in \ct{E} together with a relation $\alpha \subseteq \NN \times A$, which is surjective:
\[ \forall a \in A \, \exists n \in \NN \, (n, a) \in \alpha. \]
The natural numbers $n$ such that $(n, a) \in \alpha$ are called the \emph{realizers} of $a$, and we will frequently write $n \in \alpha(a)$ instead of $(n, a) \in \alpha$.

A morphism \func{f}{B}{A} in \ct{E} is a morphism of assemblies $(B, \beta) \to (A, \alpha)$ if the statement
\begin{quote}
``There is a natural number $r$ such that for all $b$ and $n \in \beta(b)$, the expression $r(n)$ is defined and $r(n) \in \beta(fb)$.''
\end{quote}
is valid in the internal logic of \ct{E} (note that this makes sense, as the internal logic of \ct{E} is a version of {\bf HA}, and therefore strong enough to do all basic recursion theory). A number $r$ witnessing the above statement is said to \emph{track} (or \emph{realize}) the morphism $f$. The resulting category will be denoted by $\Asm_\ct{E}$, or simply $\Asm$.
\end{defi}

We investigate the structure of the category $\Asm_\ct{E}$.

\emph{$\Asm_\ct{E}$ has finite limits.} The terminal object is $(1, \eta)$, where $1 = \{ * \}$ is a one-point set and $n \in \eta(*)$ for every $n$. The pullback $(P, \pi)$ of $f$ and $g$ as in
\diag{(P, \pi) \ar[r] \ar[d] & (B, \beta) \ar[d]^f \\
(C, \gamma) \ar[r]_g & (A, \alpha) }
can be obtained by putting $P = B \times_A C$ and
\[ n \in \pi(b, c) \Leftrightarrow n_0 \in \beta(b) \mbox{ and } n_1 \in \gamma(c) . \]

\emph{Covers in $\Asm_\ct{E}$.} A morphism \func{f}{(B, \beta)}{(A, \alpha)} is a cover if, and only if, the statement
\begin{quote}
``There is a natural number $s$ such that for all $a \in A$ and $n \in \alpha(a)$ there exists a $b \in B$ with $f(b) = a$ and such that the expression $s(n)$ is defined and $s(n) \in \beta(b)$.''
\end{quote}
holds in the internal logic of \ct{E}. From this it follows that covers are stable under pullback in \Asm.

\emph{$\Asm_\ct{E}$ has images.} A morphism \func{f}{(B, \beta)}{(A, \alpha)} is monic in $\Asm$ if, and only if, the underlying morphism \func{f}{B}{A} is monic in \ct{E}. (This means that if $(R, \rho)$ is a subobject of $(A, \alpha)$, then $R$ is also a subobject of $A$.) Hence the image $(I, \iota)$ of a map \func{f}{(B, \beta)}{(A, \alpha)} as in
\diag{(B, \beta) \ar[rr]^f \ar@{->>}[dr]_e & & (A, \alpha) \\
& (I, \iota) \ar@{ >->}[ur]_m & }
can be obtained by letting $I \subseteq A$ be the image of $f$ in \ct{E}, and
\[ n \in \iota(a) \Leftrightarrow \exists b \in B \, p(b)=a \mbox{ and } n \in \beta(b). \]
One could also write: $\iota(a) = \bigcup_{b \in p^{-1}(a)} \beta(b)$. We conclude that $\Asm$ is regular.

\emph{$\Asm_\ct{E}$ is Heyting.} For any diagram of the form
\diag{ (S, \sigma) \ar@{ >->}[d] & & \\
(B, \beta) \ar[rr]_f & & (A, \alpha) }
we need to compute $(R, \rho) = \forall_f(S, \sigma)$. We first put $R_0 = \forall_f S \subseteq A$, and let $\rho \subseteq \NN \times R_0$ be defined by
\[ n \in \rho(a) \Leftrightarrow n_0 \in \alpha(a) \mbox{ and } \forall b \in f^{-1}(a), m \in \beta(b) \, ( \, n_1(m) \downarrow \mbox{ and } n_1(m) \in \sigma(b) \, ). \]
If we now put
\[ R = \{ a \in R_0 \, : \, \exists n \, n \in \rho(a) \} \]
and restrict $\rho$ accordingly, the subobject $(R, \rho)$ will be the result of universal quantifying $(S, \sigma)$ along $f$.

\emph{$\Asm_\ct{E}$ is positive.} The sum $(A, \alpha) + (B, \beta)$ is simply $(S, \sigma)$ with $S = A+B$ and
\[ n \in \sigma(s) \Leftrightarrow n \in \alpha(s) \mbox{ if } s \in A \mbox{, and } n \in \beta(s) \mbox{ if } s \in B. \]

We have proved:
\begin{prop}{asmisHeyting}
The category $\Asm_\ct{E}$ of assemblies relative to \ct{E} is a positive Heyting category.
\end{prop}

The next step is to define the display maps in the category of assemblies. The idea is that a displayed assembly is an object $(B, \beta)$ in which both $B$ and the subobject $\beta \subseteq \NN \times B$ are small. When one tries to define a family of such displayed objects indexed by an assembly $(A, \alpha)$ in which neither $A$ nor $\alpha$ needs to be small, one arrives at the concept of a standard display map. To formulate it, we need a piece of notation.
\begin{defi}{bracketing} Let $(B, \beta)$ and $(A, \alpha)$ be assemblies and \func{f}{B}{A} be an arbitrary map in \ct{E}. We construct a new assembly $(B, \beta[f])$ by putting
\[ n \in \beta[f](b) \Leftrightarrow n_0 \in \beta(b) \mbox{ and } n_1 \in \alpha(fx). \]
\end{defi}

\begin{rema}{somethaboutstdispmaps}
Note that we obtain a morphism of assemblies of the form $(B, \beta[f]) \to (A, \alpha)$, which, by abuse of notation, we will also denote by $f$. Moreover, if $f$ was already a morphism of assemblies it can now be decomposed as
\diag{ (B, \beta) \ar[r]^\cong & (B, \beta[f]) \ar[r]^f & (A, \alpha). }
\end{rema}

\begin{defi}{displayinAsm} A morphism of assemblies of the form $(B, \beta[f]) \to (A, \alpha)$ will be called a \emph{standard display map}, if both $f$ and the mono $\beta \subseteq \NN \times B$ are small in \ct{E} (the latter is equivalent to $\beta \to B$ being small, or $\beta(b)$ being a small subobject of $\NN$ for every $b \in B$). A \emph{display map} is a morphism of the form
\diag{ W \ar[r]^\cong & V \ar[r]^f & U, }
where $f$ is a standard display map. We will write $\smallmap{D}_\ct{E}$ for the class of display maps in $\Asm_\ct{E}$.
\end{defi}

\begin{lemm}{somepropstdispmaps}
\begin{enumerate}
\item Let \func{f}{(B, \beta[f])}{(A, \alpha)} be a standard display map, and \func{g}{(C, \gamma)}{(A, \alpha)} be an arbitrary morphism of assemblies. Then there is a pullback square
\diag{(P, \pi[k]) \ar[r]^h \ar[d]_k & (B, \beta[f]) \ar[d]^f \\
(C, \gamma) \ar[r]_g & (A, \alpha) }
in which $k$ is again a standard display map.
\item The composite of two standard display maps is a display map.
\end{enumerate}
\end{lemm}
\begin{proof}
(1) We set $P = B \times_A C$ (as usual), and
\[ n \in \pi(b, c) \Leftrightarrow n \in \beta(b), \]
turning $k$ into a standard display map. Moreover, this implies
\[ n \in \pi[k](b, c) \Leftrightarrow n_0 \in \beta(b) \mbox{ and } n_1 \in \gamma(c), \]
which is precisely the usual definition.

(2) Let $(C, \gamma)$, $(B, \beta)$ and $(A, \alpha)$ be assemblies in which $\gamma \subseteq \NN \times C$ and $\beta \subseteq \NN \times B$ are small monos, and \func{g}{C}{B} and \func{f}{B}{A} be display maps in \ct{E}. These data determine a composable pair of standard display maps \func{f}{(B, \beta[f])}{(A, \alpha)} and \func{g}{(C, \gamma[g])}{(B, \beta[f])}, in which
\begin{eqnarray*}
n \in \gamma[g](c) & \Leftrightarrow & n_0 \in \gamma(c) \mbox{ and } n_1 \in \beta[f](gc) \\
& \Leftrightarrow & n_0 \in \gamma(c) \mbox{ and } (n_1)_0 \in \beta(gc) \mbox{ and } (n_1)_1 \in \gamma(fgc).
\end{eqnarray*}
So its composite can be written as
\diag{ (C, \gamma[g]) \ar[r]^{\cong} & (C, \delta[fg]) \ar[r]^{fg} & (A, \alpha),}
where we have defined $\delta \subseteq \NN \times C$ by
\[ n \in \delta(c) \Leftrightarrow n_0 \in \gamma(c) \mbox{ and } n_1 \in \beta(gc). \]
\end{proof}

\begin{coro}{somepropdispmaps}
Display maps are stable under pullback and closed under composition.
\end{coro}
\begin{proof}
Stability of display maps under pullback follows immediately from item 1 in the previous lemma. To show that they are also closed under composition, we observe first that a morphism $f$ which can be written as a composite
\diag{ W \ar[r]^h & V \ar[r]^g_\cong & U, }
where $h$ is a standard display map and $g$ is an isomorphism, is a display map. For it follows from the previous lemma that there exists a pullback square
\diag{ Q \ar[r]^p_\cong \ar[d]_q & W \ar[d]^h \\
U \ar[r]_{g^{-1}}^\cong & V }
in which $q$ is a standard display map. Therefore $f = qp^{-1}$ is a display map. Now the result follows from the lemma above.
\end{proof}

\begin{prop}{axiomsfordispinasm}
The class of display maps in the category $\Asm_\ct{E}$ of assemblies as defined above satisfies the axioms {\bf (A1), (A3-5), (A7-9)}, and {\bf (A10)} for a class of display maps, as well as {\bf (NE)} and {\bf (NS)} (see Appendix B).
\end{prop}
\begin{proof}
{\bf (A1)} We have proved pullback stability in the corollary above.

{\bf (A3)} It is easy to see that the sum of two standard display maps can be chosen to be a standard display map again. From this {\bf (A3)} follows.

{\bf (A4)} It is also easy to see that the maps $0 \rTo 1, 1 \rTo 1$ and $1+1 \rTo 1$ are standard display maps.

{\bf (A5)} Closure of display maps under composition we showed in the corollary above.

{\bf (A7)} We postpone the proof of the fact that the display maps satisfy the collection axiom: one will be given in a lemma below.

{\bf (A8)} We start with a diagram of the form
\diag{ (S, \sigma[i]) \ar@{ >->}[d]_i & & \\
(B, \beta[f]) \ar[rr]_f & & (A, \alpha), }
in which both maps are standard display maps (this is sufficient to establish the general case). We compute $(R, \rho) = \forall_f(S, \sigma)$: first we put $R_0 = \forall_f S \subseteq A$, and let $\rho \subseteq \NN \times R_0$ be defined by
\begin{eqnarray*}
 n \in \rho(a) & \Leftrightarrow & n_0 \in \alpha(a) \mbox{ and } \\
& & \forall b \in f^{-1}(a), m \in \beta[f](b) \, ( \, n_1(m) \downarrow \mbox{ and } n_1(m) \in \sigma[i](b) \, ).
\end{eqnarray*}
Furthermore, we set
\[ R = \{ a \in R_0 \, : \, \exists n \, n \in \rho(a) \} \]
and denote by $j$ the inclusion $R \subseteq A$. By restricting $\rho$, the subobject $(R, \rho)$ will be the result of the universal quantifying $(S, \sigma)$ along $f$. In the particular case we are in, this can be done differently.

We define $\tau \subseteq \NN \times R_0$ by
\[ n \in \tau(a) \Leftrightarrow \forall a \in f^{-1}(a), m \in \beta(b) \, ( \, n_1(m) \downarrow \mbox{ and } n_1(m) \in \sigma(b) \, ). \]
Note that we have a bounded formula on the right (using that $\NN$ is small). Now one can show that
\[ R = \{ a \in R_0 \, : \, \exists n \, n \in \tau(a) \}, \]
from which it follows that $j$ is a display map (again using that $\NN$ is small). Furthermore, one can prove that the identity is an isomorphism of assemblies
\[ (R, \rho) \cong (R, \tau[j]), \]
from which it follows that $(R, \rho) \to (A, \alpha)$ is a display map.

{\bf (A9)} The product of an assembly $(X, \chi)$ with itself can be computed by taking $(X \times X, \chi \times \chi)$, where
\[ n \in (\chi \times \chi)(x, y) \Leftrightarrow n_0 \in \chi(x) \mbox{ and } n_1 \in \chi(y). \]

This means that by writing \func{\Delta}{X}{X \times X} for the diagonal map in \ct{E}, the diagonal map in assemblies can be decomposed as follows
\diag{ (X, \chi) \ar[r]^(.44)\cong & (X, \mu[\Delta]) \ar[r]^(.42)\Delta & (X, \chi) \times (X, \chi),}
where $\mu \subseteq \NN \times X$ is the relation defined by
\[ n \in \mu(x) \Leftrightarrow \mbox{Always.} \]

{\bf (A10)} We need to show that for a display map $f$, if $f = me$ with $m$ a mono and $e$ a cover, then also $m$ is display. Without loss of generality, we may assume that $f$ is a standard display map \func{f}{(B, \beta[f])}{(A, \alpha)}. From \refprop{asmisHeyting}, we know that we can compute its image $(I, \iota)$ by putting $I = {\rm Im}(f)$ and
\[ n \in \iota(a) \Leftrightarrow \exists b \in f^{-1}(a) \,  n \in \beta(b). \]
As the formula on the right is bounded, the map \func{m}{(I, \iota)}{(A, \alpha)} can be decomposed as an isomorphism followed by a standard display map:
\diag{ (I, \iota) \ar[r]^\cong & (I, \iota[m]) \ar[r]^m & (A, \alpha).}

{\bf (NE)} and {\bf (NS)} The nno in assemblies is the pair consisting of $\NN$ together with the diagonal $\Delta \subseteq \NN \times \NN$.
\end{proof}

We will use the proof that the display maps in assemblies satisfy collection to illustrate a technique that does not really save an enormous amount of labour in this particular case, but will be very useful in more complicated situations.

\begin{defi}{partasm} An assembly $(A, \alpha)$ will be called \emph{partitioned}, if
\[ n \in \alpha(a), m \in \alpha(a) \Rightarrow n = m. \]
From this it follows that $\alpha$ can be considered as a morphism $A \to \NN$.
\end{defi}

\begin{lemm}{propofpartasm}
\begin{enumerate}
\item Every assembly is covered by a partitioned assembly. Hence every morphism between assemblies is covered by a morphism between partitioned assemblies.
\item A morphism \func{f}{(B, \beta)}{(A, \alpha)} between partitioned assemblies is display iff $f$ is small in \ct{E}.
\item Every display map between assemblies is covered by a display map between partitioned assemblies.
\end{enumerate}
\end{lemm}
The definitions of the notions of a covering square and the covering relation between maps from \cite{bergmoerdijk07b} are recalled in Appendix B.

\begin{proof}
(1) If $(A, \alpha)$ is an assembly, $\alpha$ can be considered as a partitioned assembly with $n$ realizing an element $(m, a) \in \alpha$ iff $n = m$. This partitioned assembly covers $(A, \alpha)$.

(2) By definition it is the case that every display map between partitioned assemblies has an underlying map which is small. Conversely, if $(B, \beta)$ is a partitioned assembly, the set $\beta(b)$ is a singleton, and therefore small. So the decomposition
\diag{ (B, \beta) \ar[r]^\cong & (B, \beta[f]) \ar[r]^f & (A, \alpha). }
shows that $f$ is a display map, if the underlying morphism is small.

(3) If \func{f}{(B, \beta[f])}{(A, \alpha)} is a standard display map between assemblies, then
\diag{ \beta[f] \ar[r] \ar[d]_f & (B, \beta[f]) \ar[d]^f \\
\alpha \ar[r] & (A, \alpha) }
is a covering square with a display map between partitioned assemblies on the left.
\end{proof}

\begin{lemm}{collfordispmaps}
The class of display maps in the category $\Asm_\ct{E}$ of assemblies satisfies the collection axiom {\bf (A7)}.
\end{lemm}
\begin{proof}
In view of the lemma above, the general case follows by considering a display map \func{f}{(B, \beta)}{(A, \alpha)} between partitioned assemblies and a cover \func{q}{(E, \eta)}{(B, \beta)}. The fact that $q$ is a cover means that there exists a natural number $t$ such that
\begin{equation}\label{propoft}
\begin{array}{l}
\mbox{``For all $b \in B$, the expression $t(\beta b)$ is defined, and} \\ \mbox{there exists an $e \in E$ with $q(e) = b$ and $t(\beta b) \in \eta(e)$.''}
\end{array}
\end{equation}
We will collect all those natural numbers in an object
\[ T = \{ t \, : \, t \mbox{ is a natural number with the property defined above} \}, \]
which can be turned into a partitioned assembly by putting $\theta(t) = t$.
From (\ref{propoft}) it follows that $T$ is an inhabited set, and that for
\[ E' = \{ (e, b, t) \, : \, q(e) = b, t(\beta b)\downarrow, t(\beta b) \in \eta(e) \}, \]
the projection \func{p}{E'}{B \times T} will be a cover. So we can apply collection in \ct{E} to obtain a covering square
\diag{ D \ar[d]_g \ar[r]^{h} & E' \ar@{->>}[r]^(.4){p} & B \times T \ar[d]^{f \times T} \\
C \ar@{->>}[rr]_{k} & & A \times T,}
where $g$ is a small map. It is not so hard to see that from this diagram in \ct{E}, we obtain two covering squares in the category of assemblies
\diag{ (D, \delta )  \ar[d]_g \ar@{->>}[r]^(.36){ph} & (B \times T, \beta \times \tau) \ar@{->>}[r] \ar[d]_{f \times T} & (B, \beta) \ar[d]^f \\
(C, \gamma ) \ar@{->>}[r]_(.36)k & (A \times T, \alpha \times \tau) \ar@{->>}[r] & (A, \alpha), }
where we have set
\begin{eqnarray*}
\gamma(c) & = & (\alpha \times \tau)(kc) \mbox{ and } \\
\delta(d) & = & (\beta \times \tau)(phd).
\end{eqnarray*}
Since $g$ is a display map between partitioned assemblies, we only need to verify that the map $(D, \delta) \to (B, \beta)$ along the top of the above diagram factors as
\diag{ (D, \delta ) \ar[r]^l & (E, \eta) \ar@{->>}[r]^q & (B, \beta). }
We set $l = \pi_1 h$, because one can show that this morphism is tracked, as follows. If $h(d) = (e, t, b)$ for some $d \in D$, then the realizer of $d$ consists of a code $n$ for the partial recursive function $t$, together with the realizer $\beta b$ of $b$. By definition of $E'$, the expression $n(\beta b)$ is defined and a realizer for $e = (\pi_1 h)(d) = l(d)$.
\end{proof}

\section{The predicative realizability category}

Let us recall from \cite{carboni95} the construction of the (ordinary) exact completion $\ct{F}_{ex/reg}$ of a Heyting category \ct{F}. Objects of $\ct{F}_{ex/reg}$ are the equivalence relations in \ct{F}, which we will denote by $X/R$ when $R \subseteq X \times X$ is an equivalence relation. Morphisms from $X/R$ to $Y/S$ are \emph{functional relations}, i.e., subobjects $F \subseteq X \times Y$ satisfying the following statements in the internal logic of \ct{F}:
\begin{displaymath}
\begin{array}{l}
\forall x \, \exists y \, F(x, y), \\
xRx' \land ySy' \land F(x,y) \rightarrow F(x', y'), \\
F(x, y) \land F(x, y') \rightarrow ySy'.
\end{array}
\end{displaymath}
There is a functor \func{{\bf y}}{\ct{F}}{\ct{F}_{\rm ex/reg}} sending an object $X$ to $X/\Delta_X$, where $\Delta_X$ is the diagonal $X \to X \times X$. This functor is a full embedding preserving the structure of a Heyting category. When \smallmap{T} is a class of display maps in \ct{F}, one can identify the following class of maps in $\ct{F}_{\rm ex/reg}$:
\begin{eqnarray*}
g \in \overline{\smallmap{T}} & \Leftrightarrow & g \mbox{ is covered by a morphism of the form } {\bf y}f \mbox{ with } f \in \smallmap{T}.
\end{eqnarray*}
We refer to the pair $(\overline{\ct{F}}, \overline{\smallmap{T}})$, consisting of the full subcategory $\overline{\ct{F}}$ of $\ct{F}_{ex/reg}$ containing those equivalence relations \func{i}{R}{X \times X} for which $i$ belongs to $\overline{\smallmap{T}}$, together with $\overline{\smallmap{T}}$, as the \emph{exact completion} of the pair $(\ct{F}, \smallmap{T})$. In \cite{bergmoerdijk07b} we proved the following result for such exact completions:
\begin{theo}{excomplpred} {\rm \cite{bergmoerdijk07b}}
If $(\ct{F},\smallmap{T})$ is a category with a representable class of display maps satisfying {\bf ($\Pi$E), (WE)} and {\bf(NS)}, then its exact completion $(\overline{\ct{F}}, \overline{\smallmap{T}})$ is a predicative category with small maps.
\end{theo}
In the rest of the section, we let $(\ct{E}, \smallmap{S})$ be a predicative category with small maps. For such a category we have constructed and studied the pair $(\Asm_\ct{E}, \smallmap{D}_\ct{E})$ consisting of the category of assemblies and its display maps. We now define $(\Eff_\ct{E}, \smallmap{S}_\ct{E})$ as the exact completion of $(\Asm_\ct{E}, \smallmap{D}_\ct{E})$ and prove our main theorem (\reftheo{effascatwsmallmaps}) as an application of \reftheo{excomplpred}. Much of the work has already been done in Section 2. In fact, \refprop{axiomsfordispinasm} shows that the only thing that remains to be shown are the representability and the validity of axioms {\bf ($\Pi$E)} and {\bf (WE)} (see Appendix B).
\begin{prop}{reprfordispl}
The class of display maps in the category $\Asm_\ct{E}$ of assemblies is representable.
\end{prop}
\begin{proof}
Let \func{\pi}{E}{U} be the representation for the small maps in \ct{E}. We define two partitioned assemblies $(T, \tau)$ and $(D, \delta)$ by
\begin{eqnarray*}
T & = & \{ (u \in U, \func{p}{E_u}{\NN}) \}, \\
\tau(u,p) & = & 0, \\
D & = & \{ (u \in U, \func{p}{E_u}{\NN}, e \in E_u ) \}, \\
\delta(u, p, e) & = & pe.
\end{eqnarray*}
Clearly, the projection \func{\rho}{(D, \delta)}{(T, \tau)} is a display map, which we will now show is a representation.

Assume \func{f}{(B, \beta)}{(A, \alpha)} is a display map between partitioned assemblies (in view of \reflemm{propofpartasm} it is sufficient to consider this case). Since $f$ is also a display map in \ct{E} we find a diagram of the form
\diag{B \ar[d]_f & N \ar[d]^s \ar[r]^k \ar@{->>}[l]_l & E \ar[d]^{\pi} \\
A &  M \ar[r]_g \ar@{->>}[l]^h & U, }
where the left square is covering and the right one a pullback. This induces a similar picture
\diag{(B, \beta) \ar[d]_f & (N, \nu) \ar[d]^s \ar[r]^{k'} \ar@{->>}[l]_l & (D, \delta) \ar[d]^{\rho} \\
(A, \alpha) &  (M, \mu) \ar[r]_{g'} \ar@{->>}[l]^h & (T, \tau) }
in the category of assemblies, where we have set:
\begin{eqnarray*}
g'(m) & = & (gm, \func{\beta l k^{-1}}{E_{gm}}{\NN}), \\
\mu(m) & = & \alpha h(m) \mbox{, so } h \mbox{ is tracked and a cover}, \\
k'(n) & = & (g's(n), kn), \\
\nu(n) & = & \langle \mu sn, \delta k' n \rangle \mbox{, so the righthand square is a pullback.}
\end{eqnarray*}
Here $g'$ is well-defined, because $N$ is a pullback and therefore the map $k$ induces for every $m \in M$ an isomorphism
\diag{ N_m \ar[r]^(.46)k_(.46)\cong & E_{gm}.}
It remains to see that $l$ is tracked, and the that left hand square is a quasi-pullback. For this, one unwinds the definition of $\nu$:
\begin{eqnarray*}
\nu(n) & = & \langle \mu sn, \delta k' n \rangle \\
& = & \langle \mu sn, \delta(g's(n), kn) \rangle \\
& = & \langle \mu sn, \delta(gs(n), \beta l k^{-1}, kn) \rangle \\
& = & \langle \mu sn, \beta l k^{-1} kn \rangle \\
& = & \langle \mu sn, \beta l n \rangle.
\end{eqnarray*}
From this description of $\nu$, we see that $l$ is indeed tracked (by the projection on the second coordinate). To see that the square is a quasi-pullback, one uses first of all that it is a quasi-pullback in \ct{E}, and secondly that the realizers for an element in $\NN$ are the same as that of its image in the pullback $(M \times_A B, \mu \times \beta)$ along the canonical map to this object.
\end{proof}

\begin{prop}{functypesinasm}
The display maps in the category $\Asm_\ct{E}$ of assemblies are exponentiable, i.e., satisfy the axiom {\bf ($\Pi$E)}. Moreover, if {\bf ($\Pi$S)} holds in \ct{E}, then the it holds for the display maps in $\Asm_\ct{E}$ as well.
\end{prop}
\begin{proof}
Let \func{f}{(B, \beta[f])}{(A, \alpha)} be a standard display map and \func{g}{(C, \gamma)}{(A, \alpha)} an arbitrary map with the same codomain. For showing the validity of {\bf ($\Pi$E)} it suffices to prove that the exponential $g^f$ exists in the slice over $(A, \alpha)$.

Since $f$ is small, one can form the exponential $g^f$ in $\ct{E}/A$, whose typical elements are pairs $(a \in A, \func{\phi}{B_a}{C_a})$. If we set
\begin{eqnarray*}
n \in \eta(a, \phi) & \Leftrightarrow & n_0 \in \alpha(a) \mbox{ and } (\forall b \in B_a, m \in \beta(b)) \, [n_1(m) \downarrow \mbox{ and } n_1(m) \in \gamma(\phi b)], \\
E & = & \{ (a, \phi) \in f^g \, : \, (\exists n \in \NN) \, [n \in \eta(a, \phi)] \},
\end{eqnarray*}
the assembly $(E, \eta)$ with the obvious projection $p$ to $(A, \alpha)$ is the exponential $g^f$ in assemblies. This shows validity of {\bf ($\Pi$E)} for the display maps in assemblies.

If \func{g}{(C, \hat{\gamma}[g])}{(A, \alpha)} is another standard display map, the exponential can also be constructed by putting
\begin{eqnarray*}
n \in \hat{\eta}(a, \phi) & \Leftrightarrow & (\forall b \in B_a, m \in \beta(b)) \, [n_1(m) \downarrow \mbox{ and } n_1(m) \in \hat{\gamma}(\phi b)], \\
\hat{E} & = & \{ (a, \phi) \in f^g \, : \, (\exists n \in \NN) \, [n \in \hat{\eta}(a, \phi)] \}.
\end{eqnarray*}
It is not hard to see that $\hat{E} = E$, and the identity induces an isomorphism of assemblies $(\hat{E}, \hat{\eta}[p]) = (E, \eta)$. This shows the stability of {\bf ($\Pi$S)}.
\end{proof}

\begin{prop}{indtypesinasm}
The display maps in the category $\Asm_\ct{E}$ of assemblies satisfy the axiom {\bf (WE)}. Moroever, if {\bf (WS)} holds in \ct{E}, then it holds for the display maps in $\Asm_\ct{E}$ as well.
\end{prop}
\begin{proof}
Let \func{f}{(B, \beta[f])}{(A, \alpha)} be a standard display map. Since {\bf (WE)} holds in \ct{E}, we can form $W_f$ in \ct{E}. On it, we wish to define the relation $\delta \subseteq \NN \times W_f$ given by
\begin{equation}\label{decoration}
\begin{array}{lcl}
n \in \delta(\mbox{sup}_a(t)) & \Leftrightarrow & n_0 \in \alpha(a) \mbox{ and } (\forall b \in f^{-1}(a), m \in \beta(b)) \,  \\
& & [n_1(m) \downarrow \mbox{ and } n_1(m) \in \delta(tb)]
\end{array}
\end{equation}
(we will sometimes call the elements $n \in \delta(w)$ the \emph{decorations} of the tree $w \in W$). It is not so obvious that we can, but for that purpose we introduce the notion of an \emph{attempt}. An attempt is an element $\sigma$ of $\spower(\NN \times W_f)$ such that
\begin{eqnarray*}
(n, \mbox{sup}_a(t)) \in \sigma & \Rightarrow & n_0 \in \alpha(a) \mbox{ and } \\
& & (\forall b \in f^{-1}(a), m \in \beta(b)) \, [n_1(m) \downarrow \mbox{ and } (n_1(m), tb) \in \sigma].
\end{eqnarray*}
If we now put
\[ n \in \delta(w) \Leftrightarrow \mbox{there exists a decoration } \sigma \mbox{ with } (n, w) \in \sigma, \]
the relation $\delta$ will have the desired property. (One direction in (\ref{decoration}) is trivial, the other is more involved and uses the collection axiom.) The W-type in the category of assemblies is now given by $(W, \delta)$ where
\[ W = \{ w \in W_f \, : \, (\exists n \in \NN) \, [n \in \delta(w)] \}. \]
This shows the vailidity of {\bf (WE)} for the display maps.

If $A$ is small and {\bf (WS)} holds in \ct{E}, then $W_f$ is small. Moreover, if $\alpha \subseteq \NN \times A$ is small, one can use the initiality of $W_f$ to define a map \func{d}{W_f}{\spower \NN} by
\begin{eqnarray*}
d(\mbox{sup}_a(t)) & = & \{ n \in \NN \, : \, n_0 \in \alpha(a) \mbox{ and }
\\ & & (\forall b \in f^{-1}(a), m \in \beta(b)) \, [n_1(m) \downarrow \mbox{ and } n_1(m) \in d(tb)] \}.
\end{eqnarray*}
Clearly, $n \in \delta(w)$ iff $n \in d(w)$, so $\delta$ is a small subobject of $\NN \times W_f$. This shows that $(W, \delta)$ is displayed, and the stability of {\bf (WS)} is proved.
\end{proof}

To summarise, we have proved the first half of \reftheo{effascatwsmallmaps}, which we phrase explicitly as:
\begin{coro}{effcatwpredsmallmaps}
If $(\ct{E}, \smallmap{S})$ is a predicative category with small maps, then so is $(\Eff_\ct{E}, \smallmap{S}_\ct{E})$.
\end{coro}

\section{Additional axioms}

To complete the proof of \reftheo{effascatwsmallmaps}, it remains to show the stability of the additional axioms {\bf (M)}, {\bf (PS)} and {\bf (F)}. That is what we will do in this (rather technical) section. We assume again that $(\ct{E}, \smallmap{S})$ is a predicative category with small maps.

\begin{prop}{sepinasm}
Assume the class of small maps in \ct{E} satisfies {\bf (M)}. Then {\bf (M)} is valid for the display maps in the category $\Asm_\ct{E}$ of assemblies and for the small maps in the predicative realizability category $\Eff_\ct{E}$ as well.
\end{prop}
\begin{proof}
Let \func{f}{(B, \beta)}{(A, \alpha)} be a monomorphism in the category of assemblies. Then the underlying map $f$ in \ct{E} is a monomorphism as well. Therefore it is small, as is the inclusion $\beta \subseteq \NN \times B$. So the morphism $f$, which factors as
\diag{ (B, \beta) \ar[r]^(.45){\cong} & (B, \beta[f]) \ar[r] & (A, \alpha), }
is a display map of assemblies.

Stability of the axiom {\bf (M)} under exact completion \cite[Proposition 6.4]{bergmoerdijk07b} shows it holds in $\Eff_\ct{E}$ as well.
\end{proof}

\begin{prop}{fullinasm}
Assume the class of small maps in \ct{E} satisfies {\bf (F)}. Then {\bf (F)} is valid for the display maps in the category $\Asm_\ct{E}$ of assemblies and for the small maps in the predicative realizability category $\Eff_\ct{E}$ as well.
\end{prop}
\begin{proof}
It is sufficient to show the validity of {\bf (F)} in the category of assemblies, for we showed the stability of this axiom under exact completion in \cite[Proposition 6.25]{bergmoerdijk07b}. So we need to find a generic \emph{mvs} in the category of assemblies for any pair of display maps \func{g}{(B, \beta)}{(A, \alpha)} and \func{f}{(A, \alpha)}{(X, \chi)}. In view of Lemma 6.23 from \cite{bergmoerdijk07b} and \reflemm{propofpartasm} above, we may without loss of generality assume that $g$ and $f$ are display maps between \emph{partitioned} assemblies.

We apply Fullness in \ct{E} to obtain a diagram of the form
\diag{ P \ar@{ >->}[r] \ar@{->>}[dr] & Y \times_X B \ar[rr] \ar[d] & & B \ar[d]^g \\
& Y \times_X A \ar[rr] \ar[d]& & A \ar[d]^f \\
& Y \ar[r]_s & X' \ar@{->>}[r]_q & X, }
where $P$ is a generic displayed \emph{mvs} for $g$. This allows us to obtain a similar diagram of partitioned assemblies
\diag{ (\tilde{P}, \tilde{\pi}) \ar@{ >->}[r] \ar@{->>}[dr] & (\tilde{Y} \times_X B, \tilde{\upsilon} \times \beta) \ar[d] \ar[rr] & & (B, \beta) \ar[d]^g \\
& (\tilde{Y} \times_X A, \tilde{\upsilon} \times \alpha) \ar[rr] \ar[d] & & (A, \alpha) \ar[d]^f \\
& (\tilde{Y}, \tilde{\upsilon}) \ar[r]_{\tilde{s}} & (X', \chi') \ar@{->>}[r]_q & (X, \chi), }
where we have set
\begin{eqnarray*}
\chi'(x') & = & \chi (qx') \mbox{ for } x' \in X', \\
\tilde{Y} & = & \{ (y, n) \in Y \times \NN \, : \, n \mbox{ realizes the statement that } P_y \to A_{qsy} \mbox{ is a cover} \} \\
& = & \{ (y, n) \in Y \times \NN \, : (\forall a \in A_{qsy})(\exists b \in B_a) \, [(y, b) \in P \mbox{ and } n (\alpha(a)) = \beta(b)] \}, \\
\tilde{\upsilon}(y, n) & = & \langle sqy, n \rangle \mbox{ for } (y, n) \in \tilde{Y}, \\
\tilde{P} & = & \tilde{Y} \times_Y P \\
& = & \{ (y, n, b) \in Y \times \NN \times B \, : \, (y, n) \in \tilde{Y}, (y, b) \in P \}, \\
\tilde{\pi}(y, n, b) & = & \langle n, \beta(b) \rangle \mbox{ for } (y, n, b) \in P.
\end{eqnarray*}
The reader should verify that:
\begin{enumerate}
\item $q$ is tracked and a cover.
\item $\tilde{s}$ is tracked and display, since $\tilde{Y}$ is defined using a bounded formula.
\item The inclusion $(\tilde{P}, \tilde{\pi}) \subseteq (\tilde{Y} \times_X B, \tilde{\upsilon} \times \beta)$ is tracked.
\item It follows from the definition of $\tilde{Y}$ that the map $(\tilde{P}, \tilde{\pi}) \to (\tilde{Y} \times_X A, \tilde{\upsilon} \times \alpha)$ is a cover.
\end{enumerate}
We will now prove that $(\tilde{P}, \tilde{\pi})$ is the generic \emph{mvs} for $g$ in assemblies.

Let $R$ be an \emph{mvs} of $g$ over $Z$, as in:
\diag{ (R, \rho) \ar@{ >->}[r]^(.37)i \ar@{->>}[dr] & (Z \times_X B, \zeta \times \beta) \ar[d] \ar[rr] & & (B, \beta) \ar[d]^g \\
& (Z \times_X A, \zeta \times \alpha) \ar[rr] \ar[d] & &  (A, \alpha) \ar[d]^f \\
& (Z, \zeta) \ar[r]_{t} & (X', \chi') \ar@{->>}[r]_q & (X, \chi). }
Since every object is covered by a partitioned assembly (see \reflemm{propofpartasm}), we may assume (without loss of generality) that $(Z, \zeta)$ is a partitioned assembly. Now we obtain a commuting square
\diag{ (\tilde{R}, \tilde{\rho}) \ar[d] \ar[r] & (R, \rho) \ar[d] \\
(\tilde{Z}, \tilde{\zeta}) \ar@{->>}[r]_d & (Z, \zeta), }
in which we have defined
\begin{eqnarray*}
\tilde{Z} & = & \{ (z, m, n) \in Z \times \NN^2 \, : \, m \mbox{ tracks } i \mbox{ and } \\ & & n \mbox{ realizes the statement that } R_z \to A_{qtz} \mbox{ is a cover} \} \\
& = & \{ (z, m, n)  \, : \, (\forall (z, b) \in R, k \in \rho(z, b)) \,  [m(k) = (\zeta \times \beta)(z, b)] \\
& & \mbox{and } (\forall a \in A_{qtz})( \exists b \in B_a) \,  [(z, b) \in R \mbox{ and } n(\alpha(a)) \in \rho(z, b) \} \\
\tilde{\zeta}(z, m, n) & = & \langle \zeta z, m, n \rangle \mbox{ for } (z, m, n) \in \tilde{Z} \\
\tilde{R} & = & \{ (z, m, n, b) \in \tilde{Z} \times B \, : \, (z, b) \in R \mbox{ and } n (\alpha(gb)) \in \rho(z, b) \} \\
\tilde{\rho}(z, m, n, b) & = & \langle \tilde{\zeta}(z, m, n), \beta(b) \rangle \mbox{ for } (z, m, n, b) \in \tilde{R}
\end{eqnarray*}
It is easy to see that all the arrows in this diagram are tracked, and the projection $(\tilde{Z}, \tilde{\zeta}) \to (Z, \zeta)$ is a cover. It is also easy to see that $(\tilde{R}, \tilde{\rho})$ is still an \emph{mvs} of $g$ in assemblies. Note also that $(\tilde{R}, \tilde{\rho})$ and $(\tilde{Z}, \tilde{\zeta})$ are partitioned assemblies.

Since the forgetful functor to \ct{E} preserves \emph{mvs}s in general, and displayed ones between partitioned assemblies, $\tilde{R}$ is also a displayed \emph{mvs} of $g$ in \ct{E}. Therefore there is a diagram of the form
\diag{ \tilde{R} \ar[d] & l^* P \ar[d] \ar[r] \ar[l] & P \ar[d] \\
\tilde{Z} & T \ar@{->>}[l]^k \ar[r]_l & Y }
in \ct{E} with $tdk = sl$. We turn $T$ into a partitioned assembly by putting $\tau(t) = \tilde{\zeta}(kt)$ for all $t \in T$.

Claim: the map \func{l}{T}{Y} factors through $\tilde{Y} \to Y$ via a map \func{\tilde{l}}{T}{\tilde{Y}} which can be tracked. Proof: if $k(t) = (z, m, n)$ and $l(t) = y$ for some $t \in T$, we set
\[ \tilde{l}(t) = (y, (m \circ n)_1), \]
where $m \circ n$ is the code of the partial recursive function obtained by composing the functions coded by $m$ with $n$. We first have to show that this is well-defined, i.e., $\tilde{l}(t) \in \tilde{Y}$. Since $P$ is an \emph{mvs} in \ct{E}, we can find for any $a \in A_{qsy}$ an element $b \in B_a$ with $(y, b) \in P$. If we take such a $b$, it follows from $P_y = P_{lt} \subseteq \tilde{R}_{kt}$, that $(z, m, n, b) \in \tilde{R}$, and therefore $n (\alpha(a)) \in  \rho(z, b)$. Moreover, it follows from the fact that $(z, m, n) \in \tilde{Z}$, that $(m \circ n)_1 (\alpha(a)) = \beta(b)$. This shows that $\tilde{l}(t) \in \tilde{Y}$. That $\tilde{l}$ is tracked is easy to see.

As a result, we obtain a diagram of the form
\diag{ (\tilde{R}, \tilde{\rho}) \ar[d] & \tilde{l}^* (\tilde{P}, \tilde{\pi}) \ar[d] \ar[r] \ar[l] & (\tilde{P}, \tilde{\pi}) \ar[d] \\
(\tilde{Z}, \tilde{\zeta}) & (T, \tau) \ar@{->>}[l]^k \ar[r]_{\tilde{l}} & (\tilde{Y}, \tilde{\upsilon}). }
Given the definitions of $\tilde{\rho}$ and $\tilde{\pi}$, one sees that $\tilde{l}^* (\tilde{P}, \tilde{\pi}) \to (\tilde{R}, \tilde{\rho})$ is tracked. This completes the proof.
\end{proof}

When it comes to the axiom {\bf (PS)} concerning power types, there seems to be no reason to believe that it will be inherited by the assemblies. But, fortunately, it will be inherited by its exact completion, and for our purposes that is just as good.
\begin{prop}{powtypesinasm}
Assume the class of small maps in \ct{E} satisfies {\bf (PS)}. Then {\bf (PS)} is valid in the realizability category $\Eff_\ct{E}$ as well.
\end{prop}
\begin{proof}
For the purpose of this proof, we introduce the notion of a weak power class object. Recall that the power class object is defined as:
\begin{defi}{powerclassobj}
By a \emph{$D$-indexed family of subobjects} of $C$, we mean a subobject $R \subseteq C \times D$. A $D$-indexed family of subobjects $R \subseteq C \times D$ will be called \emph{\smallmap{S}-displayed} (or simply \emph{displayed}), whenever the composite
\[ R \subseteq C \times D \rTo D \]
belongs to \smallmap{S}. If it exists, the \emph{power class object} $\spower X$ is the classifying object for the displayed families of subobjects of $X$. This means that it comes equipped with a displayed $\spower X$-indexed family of subobjects of $X$, denoted by $\in_X \subseteq X \times \spower X$ (or simply $\in$, whenever $X$ is understood), with the property that for any displayed $Y$-indexed family of subobjects of $X$, $R \subseteq X \times Y$ say, there exists a unique map \func{\rho}{Y}{\spower X} such that the square
\diag{ R \ar@{ >->}[d] \ar[r] & \in_X \ar@{ >->}[d] \\
X \times Y \ar[r]_(.45){\id \times \rho} & X \times \spower X}
is a pullback.
\end{defi}
If a classifying map $\rho$ as in the above diagram exists, but is not unique, we call the power class object \emph{weak}. We will denote a weak power class object of $X$ by $\wspower X$. We will show that the categories of assemblies has weak power class objects, which are moreover ``small'' (i.e., the unique map to the terminal object is a display map). This will be sufficient for proving the stability of {\bf (PS)}, as we will show in a lemma below that real power objects in the exact completion are constructed from the weak ones by taking a quotient.

Let $(X, \chi)$ be an assembly. We define an assembly $(P, \pi)$ by
\begin{eqnarray*}
P & = & \{ (\alpha \in \spower X, \func{\phi}{\alpha}{\spower \NN}) \, : \, (\forall x \in \alpha)(\exists n \in \NN) \, [n \in \phi(x)] \mbox{ and } \\
& & (\exists n \in \NN) \, (\forall x \in \alpha, m \in \phi(x)) \, [n(m) \in \chi(x)] \}, \\
\pi(\alpha, \phi) & = & \{ n \in \NN \, : \, (\forall x \in \alpha, m \in \phi(x)) \, [n(m) \in \chi(x)] \}.
\end{eqnarray*}
We claim that this assembly together with the membership relation $(E, \eta) \subseteq (X, \chi) \times (P, \pi)$ defined by
\begin{eqnarray*}
E & = & \{ (x \in X, \alpha \in \spower X, \func{\phi}{\alpha}{\spower \NN}) \, : \, (\alpha, \phi) \in P \mbox{ and } x \in \alpha \}, \\
\eta(x, \alpha, \phi) & = & \{ n \in \NN \, : \, n_0 \in \phi(x) \mbox{ and } n_1 \in \pi(\alpha, \phi) \}
\end{eqnarray*}
is a weak power object in assemblies.

For let $(S, \sigma)$ be a (standardly) displayed $(Y, \upsilon)$-indexed family of subobjects of $(X, \chi)$. This means that the underlying morphism \func{f}{S}{Y} is small, and $\sigma = \sigma[f]$ for a small relation $\sigma \subseteq \NN \times S$. Since $f$ is small, we obtain a pullback diagram of the form
\diag{ S \ar@{ >->}[d] \ar[r] & \in_X \ar@{ >->}[d] \\
X \times Y \ar[r]_(.45){\id \times s} & X \times \spower X}
in \ct{E}. We use this to build a similar diagram in the category of assemblies:
\diag{ (S, \sigma) \ar@{ >->}[d] \ar[r] & (E, \eta) \ar@{ >->}[d] \\
(X, \chi) \times (Y, \upsilon) \ar[r]_(.48){\id \times \overline{s}} & (X, \chi) \times (P, \pi),}
where we have set
\[ \overline{s}(y) = (sy, \lambda x \in sy. \sigma(x, y) ). \]
One quickly verifies that with $\overline{s}$ being defined in this way, the square is actually a pullback. This shows that $(P, \pi)$ is indeed a weak power object.

If $(X, \chi)$ is a displayed assembly, so both $X$ and $\chi \subseteq \NN \times X$ are small, and {\bf (PS)} holds in \ct{E}, then $P$ and $\pi$ are defined by bounded separation from small objects in \ct{E}. Therefore $(P, \pi)$ is a displayed object. In the exact completion, the power class object is constructed from this by quotienting this object (see the lemma below), and is therefore small.
\end{proof}

\noindent
To complete the proof of the proposition above, we need to show the following lemma, which is a variation on a result in \cite{bergmoerdijk07b}.

\begin{lemm}{redlemmforpowobj}
Let \func{{\bf y}}{(\ct{F}, \smallmap{T})}{(\overline{\ct{F}}, \overline{\smallmap{T}})} be the exact completion of a category with display maps. When $\wspower X$ is a weak power object for an $\overline{\smallmap{T}}$-small object $X$ in \ct{F}, then the power class object in $\overline{\ct{F}}$ exists; in fact, it can be obtained by quotienting ${\bf y}\wspower X$ by extensional equality.
\end{lemm}
\begin{proof}
We will drop occurences of {\bf y} in the proof.

On $\wspower X$ one can define the equivalence relation
\[ \alpha \sim \beta \Leftrightarrow (\forall x \in X)[ x \in \alpha \leftrightarrow x \in \beta]. \]
As $X$ is assumed to be $\overline{\smallmap{T}}$-small, the mono $\sim \, \subseteq \wspower X \times \wspower X$ is small, and therefore this equivalence relation has a quotient. We will write this quotient as $\spower X$ and prove that it is the power class object of $X$ in $\overline{\ct{F}}$. The elementhood relation of $\spower X$ is given by
\[ x \in [\alpha] \leftrightarrow x \in \alpha, \]
which is clearly well-defined. In particular,
\diag{\in_X \ar@{ >->}[d] \ar@{->>}[r] & \in_X \ar@{ >->}[d] \\
X \times \wspower X \ar@{->>}[r]_(.48){X \times q} & X \times \spower X }
is a pullback.

Let $U \subseteq X \times I \rTo I$ be an $\overline{\smallmap{T}}$-displayed $I$-indexed family of subobjects of $X$. We need to show that there is a unique map \func{\rho}{I}{\spower X} such that $(\id \times \rho)^* \in_X = U$.

Since $U \rTo I \in \overline{\smallmap{T}}$, there is a map $V \rTo J \in \smallmap{T}$ such that the outer rectangle in
\diag{V \ar[d]_f \ar@{->>}[r] & U \ar@{ >->}[d] \\
X \times J \ar[r] \ar[d] & X \times I \ar[d] \\
J \ar@{->>}[r]_p & I,}
is a covering square. Now also $\func{f}{V}{X \times J} \in \smallmap{T}$, and by replacing $f$ by its image if necessary and using the axiom {\bf (A10)}, we may assume that the top square (and hence the entire diagram) is a pullback and $f$ is monic.

So there is a map \func{\sigma}{J}{\wspower X} in \ct{E} with $(\id \times \sigma)^* \in_X = U$, by the ``universal'' property of $\wspower X$ in \ct{E}. As
\[ pj = pj' \Rightarrow V_j = V_{j'} \subseteq X \Rightarrow \sigma(j) \sim \sigma(j')  \]
for all $j, j' \in J$, the map $q \sigma$ coequalises the kernel pair of $p$. Therefore there is a map \func{\rho}{I}{\spower X} such that $\rho p = q\sigma$:
\diag{V \ar@{ >->}[d]_f \ar@{->>}[r] & U \ar@{ >->}[d] \ar[r] & \in_X \ar@{ >->}[d] \\
X \times J \ar@{->>}[r] \ar[d] & X \times I \ar[d] \ar[r] & X \times \spower X \ar[d] \\
J \ar@{->>}[r]^p \ar@/_/@<-1ex>[rr]_{q\sigma} & I \ar[r]^{\rho} & \spower X. }
The desired equality $(\id \times \rho)^* \in_X = U$ now follows. The uniqueness of this map follows from the definition of $\sim$.
\end{proof}

\noindent
The proof of this proposition completes the proof of our main result, \reftheo{effascatwsmallmaps}.

\section{Realizability models for set theory}

\reftheo{existmodelsetth} and \reftheo{effascatwsmallmaps} together imply that for any predicative category with small maps $(\ct{E}, \smallmap{S})$, the category $(\Eff_\ct{E}, \smallmap{S}_\ct{E})$ will contain a model of set theory. As already mentioned in the introduction, many known constructions of realizability models of intuitionistic (or constructive) set theory can be viewed as special cases of this method. In addition, our result also shows that these constructions can be performed inside weak metatheories such as {\bf CZF}, or inside other sheaf or realizability models.

To illustrate this, we will work out one specific example, the realizability model for {\bf IZF} described in McCarty \cite{mccarty84} (we will comment on other examples in the remark closing this section). To this end, let us start with the category $\Sets$ and fix an inaccessible cardinal $\kappa \gt \omega$. The cardinal $\kappa$ can be used to define a class of small maps \smallmap{S} in $\Sets$ by declaring a morphism to be small, when all its fibres have cardinality less than $\kappa$ (these will be called the $\kappa$-small maps). Because the axiom {\bf (M)} then holds both in \ct{E} and the category of assemblies, the exact completion $\overline{\Asm}$ of the assemblies is really the ordinary exact completion, i.e., the effective topos. This means we have defined a class of small maps in the effective topos. We will now verify that this is the same class of small maps as defined in \cite{joyalmoerdijk95}.

\begin{lemm}{coincidingsmmaps}
The following two classes of small maps in the effective topos coincide:
\begin{itemize}
\item[(i)] Those covered by a map $f$ between partitioned assemblies for which the underlying map in \ct{E} is $\kappa$-small (as in \cite{joyalmoerdijk95}).
\item[(ii)] Those covered by a display map $f$ between assemblies (as above).
\end{itemize}
\end{lemm}
\begin{proof}
Immediate from \reflemm{propofpartasm}, and the fact that the covering relation is transitive.
\end{proof}

By the general existence result, the effective topos contains a model of {\bf IZF} which we will call $V$.

\begin{prop}{princinmodelforizf}
In $V$ the following principles hold: {\bf (AC), (RDC), (PA), (MP), (CT)}. Moreover, $V$ is uniform, and hence also {\bf (UP)}, {\bf (UZ)}, {\bf (IP)} and {\bf (IP$_\omega$)} hold.
\end{prop}
\begin{proof}
The Axioms of Countable and Relativised Dependent Choice hold in $V$, because they hold in the effective topos (recall the remarks on the relation between truth in $V$ and truth in the surrounding category from the introduction; in particular, that ${\rm Int}(\NN) \cong \omega$). The same applies to Markov's Principle and Church's Thesis (for Church's thesis it is also essential that the model $V$ and the effective topos agree on the meaning of the $T$- and $U$-predicates).

The Presentation Axiom holds, because (internally in \Eff) every small object is covered by a small partitioned assembly (see the Lemma above), and the partitioned assemblies are internally projective in $\Eff$.

The Uniformity Principle, Unzerlegbarkeit and the Independence of Premisses principles are immediate consequences of the fact that $V$ is uniform (of course, Unzerlegbarkeit follows immediately the Uniformity Principle; note that for showing that the principles of {\bf (IP)} and {\bf (IP$_\omega$)} hold, we use classical logic in the metatheory).

To show that $V$ is uniform, we recall from \cite{bergmoerdijk07b} that the initial \spower-algebra is constructed as a quotient of the W-type associated to a representation. In \refprop{reprfordispl}, we have seen that the representation $\rho$ can be chosen to be a morphism between (partitioned) assemblies $(D, \delta) \rTo (T, \tau)$, where $T$ is uniform (every element in $T$ is realized by 0). As the inclusion of $\Asm$ in $\Eff$ preserves W-types, the associated W-type might just as well be computed in the category of assemblies. Therefore it is constructed as in \refprop{indtypesinasm}: for building the W-type associated to a map \func{f}{(B, \beta)}{(A, \alpha)}, one first builds $W(f)$ in \Sets, and defines (by transfinite induction) the realizers of an element ${\rm sup}_a(t)$ to be those natural numbers $n$ coding a pair $\langle n_0, n_1 \rangle$ such that (i) $n_0 \in \alpha(a)$ and (ii) for all $b \in f^{-1}(a)$ and $m \in \beta(b)$, the expression $n_1 (m)$ is defined and a realizer of $tb$. Using this description, one sees that a solution of the recursion equation $f = \langle 0, \lambda n. f \rangle$ realizes every tree. Hence $W(\rho)$, and its quotient $V$, are uniform in $\Eff$.
\end{proof}

We will now show that $V$ is in fact McCarty's model for {\bf IZF}. For this, we will follow a strategy different from the one in \cite{kouwenhovenvanoosten05}: we will simply ``unwind'' the existence proof for $V$ to obtain a concrete description. First, we compute $W = W(\rho)$ in assemblies (see the proof above). Its underlying set consists of well-founded trees, with every edge labelled by a natural number. Moreover, at every node the set of edges into that node should have cardinality less than $\kappa$. One could also describe it as the initial algebra of the functor $X \mapsto {\cal P}_{\kappa}(\NN \times X)$, where ${\cal P}_{\kappa}(Y)$ is the set of all subsets of $Y$ with cardinality less than $\kappa$:
\diag{ {\cal P}_{\kappa}(\NN \times W) \ar@/^/[rr]^{\rm I} & & W \ar@/^/[ll]^{\rm E}. }
Again, the realizers of a well-founded tree $w \in W$ are defined inductively: $n$ is a realizer of $w$, if for every pair $(m, v) \in {\rm E}(w)$, the expression $n(m)$ is defined and a realizer of $v$.

The next step is dividing out, internally in $\Eff$, by bisimulation:
\begin{eqnarray*}
w \sim w' & \Leftrightarrow & (\forall (m, v) \in {\rm E}(w)) \,  (\exists (m', v') \in {\rm E}(w'))[\, v \sim v'] \mbox{ and vice versa.}
\end{eqnarray*}
The internal validity of this statement should be translated in terms of realizers. To make the expression more succinct one could introduce the ``abbreviation'':
\begin{eqnarray*}
n \Vdash w' \epsilon w & \Leftrightarrow & (\exists (m, v) \in {\rm E}(w)) \, [n_0 = m \mbox{ and } n_1 \Vdash w' \sim v],
\end{eqnarray*}
so that it becomes:
\begin{eqnarray*}
n \Vdash w \sim w' & \Leftrightarrow & (\forall (m, v) \in {\rm E}(w))[\, n_0(m) \downarrow \mbox{ and } n_0 (m) \Vdash v \, \epsilon \, w'] \mbox{ and } \\ & & (\forall (m', v') \in {\rm E}(w')) \, [n_1(m') \downarrow \mbox{ and } n_1(m') \Vdash v' \, \epsilon \, w].
\end{eqnarray*}
By appealing to the Recursion Theorem, one can check that we have defined an equivalence relation on $W(\rho)$ in the effective topos (although this is guaranteed by the proof of the existence theorem for $V$). The quotient will be the set-theoretic model $V$. So, its underlying set is $W$ and its equality is given by the formula for $\sim$. Of course, when one unwinds the definition of the internal membership $\epsilon \subseteq V \times V$, one obtains precisely the formula above.
\begin{coro}{validityinVforIZF}
The following clauses recursively define what it means that a certain statement is realized by a natural number $n$ in the model $V$:
\begin{eqnarray*}
n \Vdash w' \epsilon w & \Leftrightarrow & (\exists (m, v) \in {\rm E}(w)) \, [n_0 = m \mbox{ and } n_1 \Vdash w' = v]. \\
n \Vdash w = w' & \Leftrightarrow & (\forall (m, v) \in {\rm E}(w))[\, n_0(m) \downarrow \mbox{ and } n_0(m) \Vdash v \, \epsilon \, w'] \mbox{ and } \\ & & (\forall (m', v') \in {\rm E}(w')) \, [n_1(m') \downarrow \mbox{ and } n_1(m') \Vdash v' \, \epsilon \, w]. \\
n \Vdash \phi \land \psi & \Leftrightarrow & n_0 \Vdash \phi \mbox{ and } n_1 \Vdash \psi. \\
n \Vdash \phi \lor \psi & \Leftrightarrow & n = <0,m> \mbox{ and } m \Vdash \phi \mbox{, or } n = <1, m> \mbox{ and } m \Vdash \psi. \\
n \Vdash \phi \rightarrow \psi & \Leftrightarrow & \mbox{For all } m \Vdash \phi, \mbox{we have } n \cdot m \downarrow \mbox{ and } n \cdot m \Vdash \psi. \\
n \Vdash \neg \phi & \Leftrightarrow & \mbox{There is no } m \mbox{ such that } m \Vdash \phi. \\
n \Vdash \exists x \, \phi(x) & \Leftrightarrow & n \Vdash \phi(a) \mbox{ for some } a \in V. \\
n \Vdash \forall x \, \phi(x) & \Leftrightarrow & n \Vdash \phi(a) \mbox{ for all } a \in V. \\
\end{eqnarray*}
\end{coro}
\begin{proof}
The internal logic of $\Eff$ is realizability, so the statements for the logical connectives immediately follow. For the quantifiers one uses the uniformity of $V$.
\end{proof}

\noindent
We conclude that the model is isomorphic to that of McCarty \cite{mccarty84} (based on earlier work by Friedman \cite{friedman73}).

\begin{rema}{manyvariations}
There are many variations and extensions of the construction just given, some of which we already alluded to in the introduction. First of all, instead of working with a inaccessible cardinal $\kappa$, we can also work with the category of classes in G\"odel-Bernays set theory, and call a map small if its fibres are sets. (The slight disadvantage of this approach is that one cannot directly refer to the effective topos, but has to build up a version of that for classes first.)

More generally, one can of course start with \emph{any} predicative category
with a class of small maps $(\ct{E}, \smallmap{S})$. If $(\ct{E},\smallmap{S})$ satisfies condition {\bf (F)}, then so will its realizability extension, and by \reftheo{existmodelsetth}, this will produce models of {\bf CZF} rather than {\bf IZF}. For example, if we take for $(\ct{E},\smallmap{S})$ the syntactic category with small maps associated to the the theory {\bf CZF}, then one obtains Rathjen's syntactic version of McCarty's model \cite{rathjen06}.

Alternatively (or, in addition), one can also replace number realizability by realizability for an arbitrary small partial combinatory algebra \pca{A} internal to \ct{E}.  Very basic examples arise in this way, already in the ``trivial'' case where \ct{E} is the topos of sheaves on the Sierpinski space, in which case an internal pca \pca{A} can be identified with a suitable map between pca's. The well-known Kleene-Vesley realizability \cite{kleenevesley65} is in fact a special case of this construction.  More generally, one can start with a predicative category with small maps $(\ct{E},\smallmap{S})$ and intertwine the construction of \reftheo{effascatwsmallmaps} with a similar result for sheaves, announced in \cite{bergmoerdijk07a} and discussed in detail in Part III of this series \cite{bergmoerdijk07d}:
\begin{theo}{sheavesresult} {\rm \cite{bergmoerdijk07a}}
Let $(\ct{E},\smallmap{S})$ be a predicative category with small
maps satisfying {\bf ($\Pi$S)}, and \ct{C} a small site with a basis in \ct{E}. Then the category of sheaves $\shct{\ct{E}}{\ct{C}}$ carries a natural class of maps $\smallmap{S}_{\ct{E}}[\ct{C}]$, such that the pair $(\shct{\ct{E}}{\ct{C}}, \smallmap{S}_\ct{E}[\ct{C}])$ is again a predicative category with small maps satisfying {\bf ($\Pi$S)}. Moreover, this latter pair satisfies {\bf (M)}, {\bf (F)} or {\bf (PS)}, respectively, whenever the pair $(\ct{E}, \smallmap{S})$ does.
\end{theo}
Thus, if \ct{C} is a small site in \ct{E}, and \pca{A}  is a sheaf of pca's on \ct{C}, one obtains a predicative category with small maps $(\ct{E}',\smallmap{S}') = (\Eff_{\shct{\ct{E}}{\ct{C}}}[{\cal A}], \smallmap{S}_{\shct{\ct{E}}{\ct{C}}}[{\cal A}])$, as in the case of Kleene-Vesley realizability \cite{birkedalvanoosten02}.

Any open (resp.~closed) subtopos defined by a small site in $(\ct{E}', \smallmap{S}')$ will now define another such pair $(\ct{E}'',\smallmap{S}'')$, and hence a model of {\bf IZF} or {\bf CZF} if the conditions of \reftheo{existmodelsetth} are met by the original pair $(\ct{E},\smallmap{S})$. One might refer to its semantics as ``relative realizability'' (resp.~``modified relative realizability''). It has been shown by \cite{birkedalvanoosten02} that
relative realizability \cite{awodeybirkedalscott02, streicher97} and modified realizability \cite{vanoosten97} are special cases of this, where $\shct{\ct{E}}{\ct{C}}$ is again sheaves on Sierpinski space (see also \cite{vanoosten08}).
\end{rema}

\section{A model of CZF in which all sets are subcountable}

In this section we will show that {\bf CZF} is consistent with the principle saying that all sets are subcountable (this was first shown by Streicher in \cite{streicher05}; the account that now follows is based on the work of the first author in \cite{berg06}). For this purpose, we consider again the effective topos $\Eff$ relative to the classical metatheory $\Sets$. We will show it carries another class of small maps.
\begin{lemm}{subcmaps}
The following are equivalent for a morphism \func{f}{B}{A} in $\Eff$.
\begin{enumerate}
\item In the internal logic of $\Eff$ it is true that all fibres of $f$ are  quotients of subobjects of $\NN$ (i.e., subcountable).
\item In the internal logic of $\Eff$ it is true that all fibres of $f$ are quotients of $\neg\neg$-closed subobjects of $\NN$.
\item The morphism $f$ fits into a diagram of the following shape
\diag{X \times \NN \ar[dr] & Y \ar@{ >->}[l] \ar@{->>}[r] \ar[d]^g & B \ar[d]^f \\
& X \ar@{->>}[r] & A, }
where the square is covering and $Y$ is a $\neg\neg$-closed subobject of $X \times \NN$.
\end{enumerate}
\end{lemm}
\begin{proof} Items 2 and 3 express the same thing, once in the internal logic and once in diagrammatic language. That 2 implies 1 is trivial.

$1 \Rightarrow 2$: This is an application of the internal validity in $\Eff$ of Shanin's Principle \cite[Proposition 1.7]{vanoosten94}: every subobject of $\NN$ is covered by a $\neg\neg$-closed one. For let $Y$ be a subobject of $X \times \NN$ in $\Eff / X$. Since every object in the effective topos is covered by an assembly, we may just as well assume that $X$ is an assembly $(X, \chi)$. The subobject $Y \subseteq X \times \NN$ can be identified with a function \func{Y}{X \times \NN}{{\cal P} \NN} for which there exists a natural number $r$ with the property that for every $m \in Y(x, n)$, the value $r(m)$ is defined and codes a pair $\langle k_0, k_1 \rangle$ with $k_0 \in \chi(x)$ and $k_1 = n$. One can then form the assembly $(P, \pi)$ with
\begin{eqnarray*}
P & = & \{ \, (x,n) \in X \times \NN \, : \, n \mbox{ codes a pair } \langle n_0, n_1 \rangle \mbox{ with } n_1 \in Y(x, n_0) \, \}, \\
\pi(x, n) & = & \{ \langle k_0, k_1 \rangle \, : \, k_0 \in \chi(x) \mbox{ and } k_1 = n \},
\end{eqnarray*}
which is actually a $\neg\neg$-closed subobject of $X \times \NN$. $P$ covers $Y$, clearly. The diagram
\diag{ P \ar@{->>}[rr] \ar@{ >->}[rd] & & Z \ar@{ >->}[dl] \\
       & X \times \NN & }
does not commute, but composing with the projection $X \times \NN \rTo X$ it does.
\end{proof}

Let \smallmap{T} be the class of maps having any of the equivalent properties in this lemma.

\begin{rema}{nomenclatureforsubc}
The morphisms belonging to \smallmap{T} were called ``quasi-modest'' in \cite{joyalmoerdijk95} and ``discrete'' in \cite{hylandrobinsonrosolini90}. In the latter the authors prove another characterisation of \smallmap{T} due to Freyd: the morphisms belonging to \smallmap{T} are those fibrewise orthogonal to the subobject classifier $\Omega$ in $\Eff$ (Theorem 6.8 in \emph{loc.cit.}).
\end{rema}

\begin{prop}{quasimodaresmallmaps} {\rm \cite[Proposition 5.4]{joyalmoerdijk95}}
The class \smallmap{T} is a representable class of small maps in $\Eff$ satisfying {\bf (M)} and {\bf (NS)}.
\end{prop}
\begin{proof}
To show that \smallmap{T} is a class of small maps, it is convenient to regard $\smallmap{T}$ as $\smallmap{D}^{\rm cov}$ (the class of maps covered by elements of \smallmap{D}), where $\smallmap{D}$ consists of those maps \func{g}{Y}{X} for which $Y$ is a $\neg\neg$-closed subobject of $X \times \NN$.
It is clear that \smallmap{D} satisfies axioms {\bf (A1, A3-5)} for a class of display maps, and {\bf (NS)} as well (for {\bf (A5)}, one uses that there is an isomorphism $\NN \times \NN \cong \NN$ in $\Eff$). It also satisfies axiom {\bf (A7)}, because all maps \func{g}{Y}{X} in \smallmap{D} are choice maps, i.e., internally projective as elements of $\Eff / X$. The reason is that in $\Eff$ the partitioned assemblies are projective, and every object is covered by a partitioned assembly. So if $X'$ is some partitioned assembly covering $X$, then also $X' \times \NN$ is a partitioned assembly, since $\NN$ is a partitioned assembly and partitioned assemblies are closed under products. Moreover, $Y \times_X X'$ as a $\neg\neg$-closed subobject of $X' \times \NN$ is also a partitioned assembly. From this it follows that $g$ is internally projective. A representation $\pi$ for \smallmap{D} is obtained via the pullback
\diag{ \in_{\NN} \ar@{ >->}[r] \ar[d]_{\pi} & \in_{\NN} \ar[d] \\
{\cal P}_{\neg\neg}(\NN) \ar@{ >->}[r] & {\cal P}(\NN). }
Furthermore, it is obvious that all monomorphisms belong to \smallmap{T}, since all the fibres of a monic map are subcountable (internally in $\Eff$).

Now it follows that \smallmap{T} is a representable class of small maps satisfying {\bf (M)} and {\bf (NS)} (along the lines of Proposition 2.14 in \cite{bergmoerdijk07b}).
\end{proof}

\begin{prop}{furtherpropquasimodmaps} {\rm \cite{berg06}}
The class \smallmap{T} satisfies {\bf (WS)} and {\bf (F)}.
\end{prop}
\begin{proof} (Sketch.)
We first observe that for any two morphisms \func{f}{Y}{X} and \func{g}{Z}{X} belonging to \smallmap{D}, the exponential $(f^g)_X \rTo X$ belongs to \smallmap{T}. Without loss of generality we may assume $X$ is a (partitioned) assembly. If $Y \subseteq X \times \NN$ and $Z \subseteq X \times \NN$ are $\neg\neg$-closed subobjects, then every function \func{h}{Y_x}{Z_x} over some fixed $x \in X$ is determined uniquely by its realizer, and so all fibres of $(f^g)_X \rTo X$ are subcountable.

To show the validity of {\bf (F)}, it suffices to show the existence of a generic \smallmap{T}-displayed \emph{mvs}s for maps \func{g}{B}{A} in \smallmap{D}, with \func{f}{A}{X} also in \smallmap{D} (in view of Lemmas 2.15 and 6.23 from \cite{bergmoerdijk07b}). Because $f$ is a choice map, one can take the object of all sections of $g$ over $X$, which is subcountable by the preceding remark.

The argument for the validity of {\bf (WS)} is similar. We use again that every composable pair of maps \func{g}{B}{A} and \func{f}{A}{X} belonging to \smallmap{T} fit into covering squares of the form
\diag{B' \ar[d]_{g'} \ar@{->>}[r] & B \ar[d]^g \\
A' \ar[r] \ar[d]_{f'} & A \ar[d]^f \\
X' \ar@{->>}[r]_p & X,}
with $g'$ and $f'$ belonging to \smallmap{D}. We may also assume that $X'$ is a (partitioned) assembly. The W-type associated to $g'$ in $\Eff / X'$ is subcountable, because every element of $W(g')_{X'}$ in the slice over some fixed $x \in X'$ is uniquely determined by its realizer. The W-type associated to $p^*g$ in the slice over $X'$ is then a subquotient of $W(g')_{X'}$ (see the proof of Proposition 6.16 in \cite{bergmoerdijk07b}), and therefore also subcountable. Finally, the W-type associated to $g$ in the slice over $X$ is also subcountable, by descent for \smallmap{T}.
\end{proof}

We will obtain a model of ${\bf CZF}$ and Full separation by considering the initial algebra $U$ for the power class functor associated to \smallmap{T}, which we will denote by ${\cal P}_t$. 
\diag{ {\cal P}_t U \ar@/^/[rr]^{\rm Int} & & U \ar@/^/[ll]^{\rm Ext}. }
In the proposition below, we show that it is not a model of {\bf IZF}, for it refutes the power set axiom.

\begin{prop}{nopowinV}
The statement that all sets are subcountable is valid in the model $U$. Therefore it refutes the power set axiom.
\end{prop}
\begin{proof}
As we explained in the introduction, the statement that all sets are subcountable follows from the fact that, in the internal logic of the effective topos, all fibres of maps belonging to \smallmap{T} are subcountable. But the principle that all sets are subcountable immediately implies the non-existence of ${\cal P}\omega$, using Cantor's Diagonal Argument. And neither does ${\cal P} 1$ when $1 = \{ \emptyset \}$ is a set consisting of only one element. For if it would, so would $({\cal P}1)^{\omega}$, by Subset Collection. But it is not hard to see that $({\cal P} 1)^{\omega}$ can be reworked into the powerset of $\omega$.
\end{proof}

\begin{prop}{principlesinV}
The choice principles {\bf (CC), (RDC), PA)} are valid in the model $U$. Moreover, as an object of the effective topos, $U$ is uniform, and therefore the principles {\bf (UP)}, {\bf (UZ)}, {\bf (IP)} and {\bf (IP$_\omega$)} hold in $U$ as well.
\end{prop}
\begin{proof} The proof is very similar to that of \refprop{princinmodelforizf}.

The Axioms of Countable and Relativised Dependent Choice $U$ inherits from the effective topos $\Eff$. To see that in $U$ every set is the surjective image of a projective set, notice that every set is the surjective image of a $\neg\neg$-closed subset of $\omega$, and these are internally projective in $\Eff$.

To show that $U$ is uniform it will suffice to point out that the representation can be chosen to be of a morphism of assemblies with uniform codomain. Then the argument will proceed as in \refprop{princinmodelforizf}. In the present case, the representation $\pi$ can be chosen to be of the form
\diag{ \in_{N} \ar@{ >->}[r] \ar[d]_{\pi} & \in_{N} \ar[d] \\
{\cal P}_{\neg\neg}(N) \ar@{ >->}[r] & {\cal P}(N). }
So therefore $\pi$ is a morphism between assemblies, where ${\cal P}_{\neg\neg}(N) = \nabla {\cal P}\mathbb{N}$, i.e. the set of all subsets $A$ of the natural numbers, with $A$ being realized by 0, say, and $\in_{N} = \{ (n, A) \, : \, n \in A \}$, with $(n, A)$ being realized by $n$. So $\pi$ is indeed of the desired form, and $U$ will be uniform. Therefore it validates the principles {\bf (UP)}, {\bf (UZ)}, {\bf (IP)} and {\bf (IP$_\omega$)}. 
\end{proof}

\begin{rema}{regextaxiom}
It follows from results in \cite{moerdijkpalmgren02} that the Regular Extension Axiom from \cite{aczelrathjen01} also holds in $U$. For in \cite{moerdijkpalmgren02}, the authors prove that the validity of the Regular Extension Axiom in $V$ follows from the axioms {\bf (WS)} and {\bf (AMC)} for \smallmap{T}. {\bf (AMC)} is the Axiom of Multiple Choice (see \cite{moerdijkpalmgren02}), which holds here because every $f \in \smallmap{T}$ fits into a covering square
\diag{Y \ar@{->>}[r] \ar[d]_g & B \ar[d]^f \\ X \ar@{->>}[r] & A,}
where \func{g}{Y}{X} is a small choice map, hence a small collection map over $X$.
\end{rema}

The model $U$ has appeared in different forms in the literature, its first appearance being in Friedman's paper \cite{friedman77}. We discuss several of its incarnations.

We have seen above that for any strongly inaccessible cardinal $\kappa \gt \omega$, the effective topos carries another class of small maps \smallmap{S}. For this class of small maps, the initial $\spower$-algebra $V$ is precisely McCarty's realizability model for {\bf IZF}. It is not hard to see that $\smallmap{T} \subseteq \smallmap{S}$, and therefore there exists a pointwise monic natural transformation ${\cal P}_{t} \Rightarrow \spower$. This implies that our present model $U$ embeds into McCarty's model.
\diag{ {\cal P}_t U \ar[r] \ar[dd]_{\rm Int} & {\cal P}_t V \ar@{ >->}[d] \\
& \spower V \ar[d]^{\rm Int} \\
U \ar@{ >->}[r] & V}

Actually, $U$ consists of those $x \in V$ that $V$ believes to be hereditarily subcountable (intuitively speaking, because $V$ and $\Eff$ agree on the meaning of the word ``subcountable'', see the introduction). To see this, write
\[ A = \{ x \in V \, : \, V \models x \mbox{ is hereditarily subcountable} \}. \]
$A$ is a ${\cal P}_t$-subalgebra of $V$, and it will be isomorphic to $U$, once one proves that is initial. It is obviously a fixed point, so it suffices to show that it is well-founded (see \cite[Theorem 7.3]{bergmoerdijk07b}). So let $B \subseteq A$ be a ${\cal P}_t$-subalgebra of $A$, and define
\[ W = \{ x \in V \, : \, x \in A \Rightarrow x \in B \}. \]
It is not hard to see that this is a $\spower$-subalgebra of $V$, so $W = V$ and $A = B$.

This also shows that principles like Church's Thesis {\bf (CT)} and Markov's Principle {\bf (MP)} are valid in $U$, since they are valid in McCarty's model $V$.

One could also unravel the construction of the initial algebra for the power class functor from \cite{bergmoerdijk07b} to obtain an explicit description, as we did in Section 5. Combining the explicit description of a representation $\pi$ in \refprop{principlesinV} with the observation that its associated W-type can be computed as in assemblies, one obtains the following description of $W = W_{\pi}$ in $\Eff$. The underlying set consists of well-founded trees where the edges are labelled by natural numbers, in such a way that the edges into a fixed node are labelled by \emph{distinct} natural numbers. So a typical element is of the form ${\rm sup}_A(t)$, where $A$ is a subset of $\NN$ and $t$ is a function $A \to W$. An alternative would be to regard $W$ as the initial algebra for the functor $X \mapsto [ \NN  \rightharpoonup X ]$, where $[ \NN \rightharpoonup X]$ is the set of partial functions from $\NN$ to $X$. The decorations (realizers) of an element $w \in W$ are defined inductively: $n$ is a realizer of ${\rm sup}_A(t)$, if for every $a \in A$, the expression $n(a)$ is defined and a realizer of $t(a)$.

We need to quotient $W$, internally in $\Eff$, by bisimulation:
\begin{eqnarray*}
\mbox{sup}_A(t) \sim \mbox{sup}_{A'}(t') & \Leftrightarrow & (\forall a \in A) \,  (\exists a' \in A') \, [ta \sim t'a'] \mbox{ and vice versa.}
\end{eqnarray*}
To translate this in terms of realizers, we again use an ``abbreviation'':
\begin{eqnarray*}
n \Vdash x \, \epsilon \, \mbox{sup}_A(t) & \Leftrightarrow & n_0 \in A \mbox{ and } n_1 \Vdash x \sim t(n_0).
\end{eqnarray*}
Then the equivalence relation $\sim \subseteq W \times W$ is defined by:
\begin{eqnarray*}
n \Vdash \mbox{sup}_A(t) \sim \mbox{sup}_{A'}(t') & \Leftrightarrow & (\forall a \in A) \, [ \, n_0(a) \downarrow \mbox{ and } n_0(a) \Vdash ta \, \epsilon \, \mbox{sup}_{A'}(t')] \mbox{ and } \\ & & (\forall a' \in A') \, [ \, n_1(a') \downarrow \mbox{ and } n_1(a') \Vdash t'a' \, \epsilon \, \mbox{sup}_{A}(t)].
\end{eqnarray*}
The quotient in $\Eff$ is precisely $U$, which is therefore the pair consisting of the underlying set of $W$ together with $\sim$ as equality. The reader should verify that the internal membership is again given by the ``abbreviation'' above.

\begin{coro}{validityinV}
The following clauses recursively define what it means that a certain statement is realized by a natural number $n$ in the model $U$:
\begin{eqnarray*}
n \Vdash x \, \epsilon \, {\rm sup}_A(t) & \Leftrightarrow & n_0 \in A \mbox{ and } n_1 \Vdash x = t(n_0). \\
n \Vdash {\rm sup}_A(t) = {\rm sup}_{A'}(t') & \Leftrightarrow & (\forall a \in A) \, [ \, n_0(a) \downarrow \mbox{ and } n_0(a) \Vdash ta \, \epsilon \, {\rm sup}_{A'}(t')] \mbox{ and } \\ & & (\forall a' \in A') \, [ \, n_1(a') \downarrow \mbox{ and } n_1(a') \Vdash t'a' \, \epsilon \, {\rm sup}_{A}(t)]. \\
n \Vdash \phi \land \psi & \Leftrightarrow & n_0 \Vdash \phi \mbox{ and } n_1 \Vdash \psi. \\
n \Vdash \phi \lor \psi & \Leftrightarrow & n = <0,m> \mbox{ and } m \Vdash \phi \mbox{, or } n = <1, m> \mbox{ and } m \Vdash \psi. \\
n \Vdash \phi \rightarrow \psi & \Leftrightarrow & \mbox{For all } m \Vdash \phi, n \cdot m \downarrow \mbox{ and } n \cdot m \Vdash \psi. \\
n \Vdash \neg \phi & \Leftrightarrow & \mbox{There is no } m \mbox{ such that } m \Vdash \phi. \\
n \Vdash \exists x \, \phi(x) & \Leftrightarrow & n \Vdash \phi(a) \mbox{ for some } a \in U. \\
n \Vdash \forall x \, \phi(x) & \Leftrightarrow & n \Vdash \phi(a) \mbox{ for all } a \in U. \\
\end{eqnarray*}
\end{coro}
From this it follows that the model is the elementary equivalent to the one used for proof-theoretic purposes by Lubarsky in \cite{lubarsky06}.

\begin{rema}{compwithstreicher}
In an unpublished note \cite{streicher05}, Streicher builds a model of {\bf CZF} based an earlier work on realizability models for the Calculus of Constructions. In our terms, his work can be understood as follows. He starts with the morphism $\tau$ in the category $\Asm$ of assemblies, whose codomain is the set of all modest sets, with a modest set realized by any natural number, and a fibre of this map over a modest set being precisely that modest set (note that this map again has uniform codomain). He proceeds to build the W-type associated to $\tau$, takes it as a universe of sets, while interpreting equality as bisimulation. One cannot literally quotient by bisimulation, for which one could pass to the effective topos.

When considering $\tau$ as a morphism in the effective topos, it is not hard to see that it is in fact another representation for the class of subcountable morphisms \smallmap{T}: for all fibres of the representation $\pi$ also occur as fibres of $\tau$, and all fibres of $\tau$ are quotients of fibres of $\pi$. Therefore the model is again the initial ${\cal P}_t$-algebra for the class of subcountable morphisms \smallmap{T} in the effective topos.
\end{rema}

\appendix

\section{Set-theoretic axioms}

Set theory is a first-order theory with one non-logical binary relation symbol $\epsilon$. Since we are concerned with constructive set theories in this paper, the underlying logic will be intuitionistic.

As is customary also in classical set theories like {\bf ZF}, we will use the abbreviations $ \exists x \epsilon a \, (\ldots)$ for $ \exists x \, (x \epsilon a \land \ldots)$, and $\forall x \epsilon a \, (\ldots)$ for $ \forall x \, (x \epsilon a \rightarrow \ldots)$. Recall that a formula is called \emph{bounded}, when all the quantifiers it contains are of one of these two forms.

\subsection{Axioms of IZF}

The axioms of {\bf IZF} are:
\begin{description}
\item[Extensionality:] $\forall x \, ( \, x \epsilon a \leftrightarrow x \epsilon b \, ) \rightarrow a = b$.
\item[Empty set:] $\exists x  \, \forall y  \, \lnot y \epsilon x $.
\item[Pairing:] $\exists x \, \forall y \, (\,  y \epsilon x \leftrightarrow y = a \lor y = b \, )$.
\item[Union:] $\exists x  \, \forall y \, ( \, y \epsilon x \leftrightarrow \exists z \epsilon a \, y \epsilon z  \, )$.
\item[Set induction:] $\forall x \, (\forall y  \epsilon x \, \phi(y) \rightarrow \phi(x)) \rightarrow \forall x \, \phi(x)$.
\item[Infinity:] $\exists a \, ( \, \exists  x \, x \epsilon a \, ) \land ( \, \forall x  \epsilon a \, \exists y \epsilon a \, x \epsilon y \, )$.
\item[Full separation:] $\exists x \,  \forall y \, ( \, y \epsilon x \leftrightarrow y \epsilon a \land \phi(y) \, ) $, for any formula $\phi$ in which $a$ does not occur.
\item[Power set:] $\exists x \, \forall y \, ( \, y \epsilon x \leftrightarrow y \subseteq a \, )$, where $y \subseteq a$ abbreviates $\forall z \, ( z \epsilon y \rightarrow z \epsilon a)$.
\item[Strong collection:] $\forall x \epsilon a \, \exists y \, \phi(x,y) \rightarrow \exists b \, \mbox{B}(x \epsilon a, y \epsilon b) \, \phi$.
\end{description}
In the last axiom, the expression
\[ \mbox{B}(x \epsilon a, y \epsilon b) \, \phi. \]
has been used as an abbreviation for $\forall x \epsilon a \, \exists y \epsilon b \, \phi \land \forall y \epsilon b \, \exists x \epsilon a \, \phi$.

\subsection{Axioms of CZF}

The set theory {\bf CZF}, introduced by Aczel in \cite{aczel78}, is obtained by replacing Full separation by Bounded separation and the Power set axiom by Subset collection:
\begin{description}
\item[Bounded separation:] $\exists x \,  \forall y \, ( \, y \epsilon x \leftrightarrow y \epsilon a \land \phi(y) \, ) $, for any bounded formula $\phi$ in which $a$ does not occur.
\item[Subset collection:] $\exists c \, \forall z \, ( \forall x \epsilon a \, \exists y \epsilon b \, \phi(x,y,z) \rightarrow \exists d \epsilon c \, \mbox{B}(x \epsilon a, y \epsilon d) \, \phi(x, y, z)) $.
\end{description}

\subsection{Constructivist principles}

In this paper we will meet the following constructivist principles associated to recursive mathematics and realizability. In writing these down, we have freely used the symbol $\omega$ for the set of natural numbers, as it is definable in both {\bf CZF} and {\bf IZF}. We also used $0$ for zero and $s$ for the successor operation.

\begin{description}
\item[Axiom of Countable Choice (CC)] \[ \forall i \epsilon \omega \, \exists x \,  \psi(i, x) \rightarrow \exists a , \func{f}{\omega}{a} \, \forall i \epsilon \omega \,  \psi(i,f(i)). \]
\item[Axiom of Relativised Dependent Choice (RDC)] \begin{displaymath}
\begin{array}{l} \phi(x_0) \, \land \, \forall x \, (\phi(x) \rightarrow \exists y \, (\psi(x,y) \, \land \, \phi(y))) \rightarrow \\ \exists a \, \exists \func{f}{\omega}{a} \, (f(0)=x_0 \land \forall i \in \omega \, \phi(f(i),f(si))).
\end{array}
\end{displaymath}
\item[Presentation Axiom (PA)] Every set is the surjective image of a projective set (where a set $a$ is projective, if every surjection $b \to a$ has a section).
\item[Markov's Principle (MP)] \[ \forall n \epsilon \omega \, [\phi(n) \lor \neg \phi(n)] \rightarrow [\neg\neg \exists n \in \omega \, \phi(n) \rightarrow \exists n \epsilon \omega \, \phi(n)].\]
\item[Church's Thesis (CT)] \[ \forall n \epsilon \omega \, \exists m \epsilon \omega \, \phi(n, m) \rightarrow \exists e \epsilon \omega \, \forall n \epsilon \omega \, \exists m,p \epsilon \omega  \, [T(e,n,p) \land U(p,m) \land \phi(n,m)] \] for every formula $\phi(u,v)$, where $T$ and $U$ are the set-theoretic predicates which numeralwise represent, respectively, Kleene's $T$ and result-extraction predicate $U$.\index{Church's Thesis}
\item[Uniformity Principle (UP)] \[ \forall x \, \exists y \epsilon \omega \, \phi(x,y) \rightarrow \exists y \epsilon \omega \, \forall x \, \phi(x,y). \]
\item[Unzerlegbarkeit (UZ)] \[ \forall x \, (\phi(x) \lor \neg \phi(x)) \rightarrow \, \forall x \, \phi \lor \forall x \, \neg \phi. \]
\item[Independence of Premisses for Sets (IP)] \[ (\lnot \theta \to \exists x \, \psi) \to \exists x \, ( \, \lnot \theta \to \psi), \]
where $\theta$ is assumed to be closed.
\item[Independence of Premisses for Numbers (IP$_\omega$)] \[ (\lnot \theta \to \exists n \epsilon \omega \, \psi) \to \exists n \epsilon \omega \, ( \, \lnot \theta \to \psi), \]
where $\theta$ is assumed to be closed.
\end{description}

\section{Predicative categories with small maps}

In the present paper, the ambient category \ct{E} is always assumed to be a \emph{positive Heyting category}. That means that \ct{E} is
\begin{enumerate}
\item[(i)] cartesian, i.e., it has finite limits.
\item[(ii)] regular, i.e., morphisms factor in a stable fashion as a cover followed by a monomorphism.
\item[(iii)] positive, i.e., it has finite sums, which are disjoint and stable.
\item[(iv)] Heyting, i.e., for any morphism \func{f}{Y}{X} the induced pullback functor \func{f^*}{{\rm Sub}(X)}{{\rm Sub}(Y)} has a right adjoint $\forall_f$.
\end{enumerate}

\begin{defi}{coveringsquare}
A diagram in \ct{E} of the form
\diag{ D \ar[d]_f \ar[r] &  C \ar[d]^g \\
B \ar[r]_p & A}
is called a \emph{quasi-pullback}, when the canonical map $D \rTo B \times_A C$ is a cover. If $p$ is also a cover, the diagram will be called a \emph{covering square}. When $f$ and $g$ fit into a covering square as shown, we say that $f$ \emph{covers} $g$, or that $g$ \emph{is covered by} $f$.
\end{defi}

A class of maps in \ct{E} satisfying the following axioms {\bf (A1-9)} will be called a \emph{class of small maps}:
\begin{description}
\item[(A1)] (Pullback stability) In any pullback square
\diag{ D \ar[d]_g \ar[r] & B \ar[d]^f \\
C \ar[r]_p & A }
where $f \in \smallmap{S}$, also $g \in \smallmap{S}$.
\item[(A2)] (Descent) If in a pullback square as above $p$ is a cover and $g \in \smallmap{S}$, then also $f \in \smallmap{S}$.
\item[(A3)] (Sums) Whenever $X \rTo Y$ and $X'\rTo Y'$ belong to \smallmap{S}, so does $X + X' \rTo Y + Y'$.
\item[(A4)] (Finiteness) The maps $0 \rTo 1, 1 \rTo 1$ and $1+1 \rTo 1$ belong to \smallmap{S}.
\item[(A5)] (Composition) $\smallmap{S}$ is closed under composition.
\item[(A6)] (Quotients) In a commuting triangle
\diag{ Z \ar[dr]_h \ar@{->>}[rr]^f & & Y \ar[dl]^g \\
& X, &  }
if $f$ is a cover and $h$ belongs to \smallmap{S}, then so does $g$.
\item[(A7)] (Collection) Any two arrows \func{p}{Y}{X} and \func{f}{X}{A} where $p$ is a cover and $f$ belongs to \smallmap{S} fit into a covering square
\diag{ Z \ar[d]_g \ar[r] & Y \ar@{->>}[r]^p & X \ar[d]^f \\
B \ar@{->>}[rr]_h & & A,}
where $g$ belongs to \smallmap{S}.
\item[(A8)] (Heyting) For any morphism \func{f}{Y}{X} belonging to \smallmap{S}, the right adjoint
\[ \func{\forall_f}{{\rm Sub}(Y)}{{\rm Sub}(X)} \]
sends small monos to small monos.
\item[(A9)] (Diagonals) All diagonals \func{\Delta_X}{X}{X \times X} belong to \smallmap{S}.
\end{description}
In case \smallmap{S} satisfies all these axioms, the pair $(\ct{E}, \smallmap{S})$ will be called a \emph{category with small maps}. Axioms {\bf (A4,5,8,9)} express that the subcategories $\smallmap{S}_X$ of ${\ct E}/X$ whose objects and arrows are both given by arrows belonging to the class \smallmap{S}, are full subcategories of ${\ct E}/X$ which are closed under all the operations of a positive Heyting category. Moreover, these categories together should form a stack on \ct{E} with respect to the finite cover topology according to the Axioms {\bf (A1-3)}. Finally, the class \smallmap{S} should satisfy the Quotient axiom {\bf (A6)} (saying that if a composition \diag{C \ar@{->>}[r] & B \ar[r] & A} belongs to \smallmap{S}, so does $B \rTo A$), and the Collection Axiom {\bf (A7)}. This axiom states that, conversely, if $B \rTo A$ belongs to \smallmap{S} and \diag{C \ar@{->>}[r] & B} is a cover (regular epimorphism), then locally in $A$ this cover has a small refinement.

The following weakening of a class of small maps will play a r\^ole as well: a class of maps satisfying the axioms {\bf (A1), (A3-5), (A7-9)}, and
\begin{description}
\item[(A10)] (Images) If in a commuting triangle
\diag{ Z \ar[dr]_f \ar@{->>}[rr]^e & & Y \ar@{ >->}[dl]^m \\
& X, &  }
$e$ is a cover, $m$ is monic, and $f$ belongs to \smallmap{S}, then $m$ also belongs to \smallmap{S}.
\end{description}
will be a called a \emph{class of display maps}.

Whenever a class of small maps (resp.~a class of display maps) \smallmap{S} has been fixed, an object $X$ will be called small (resp.~displayed), whenever the unique map from $X$ to the terminal object is small (resp.~a display map).

In this paper, we will see the following additional axioms for a class of small (or display) maps.
\begin{description}
\item[(M)] All monomorphisms belong to \smallmap{S}.
\item[(PE)] For any object $X$ the power class object $\spower X$ exists.
\item[(PS)] Moreover, for any map $\func{f}{Y}{X} \in \smallmap{S}$, the power class object $\slspower{X} (f) \rTo X$ in $\ct{E}/X$ belongs to \smallmap{S}.
\item[($\Pi$E)] All morphisms $f \in \smallmap{S}$ are exponentiable.
\item[($\Pi$S)] For any map $\func{f}{Y}{X} \in \smallmap{S}$, a functor 
\[ \func{\Pi_f}{\ct{E}/Y}{\ct{E}/X} \]
right adjoint to pullback exists and preserves morphisms in \smallmap{S}.
\item[(WE)] For all $\func{f}{X}{Y} \in \smallmap{S}$, the W-type $W_f$ associated to $f$ exists.
\item[(WS)] Moreover, if $Y$ is small, also $W_f$ is small.
\item[(NE)] \ct{E} has a natural numbers object $\NN$.
\item[(NS)] Moreover, $\NN \rTo 1 \in \smallmap{S}$.
\item[(F)] For any $\func{\phi}{B}{A} \in \smallmap{S}$ over some $X$ with $A \rTo X \in \smallmap{S}$, there is a cover \func{q}{X'}{X} and a map \func{y}{Y}{X'} belonging to $\smallmap{S}$, together with a displayed \emph{mvs} $P$ of $\phi$ over $Y$, with the following ``generic'' property: if \func{z}{Z}{X'} is any map and $Q$ any displayed \emph{mvs} of $\phi$ over $Z$, then there is a map \func{k}{U}{Y} and a cover \func{l}{U}{Z} with $yk = zl$, such that $k^* P \leq l^* Q$ as (displayed) \emph{mvs}s of $\phi$ over $U$.
\end{description}
More details are to be found in \cite{bergmoerdijk07b}.

A category with small maps $(\ct{E}, \smallmap{S})$ will be called a \emph{predicative class with small maps}, if \smallmap{S} satisfies the axioms {\bf ($\Pi$E), (WE), (NS)} and in addition:
\begin{description}
\item[(Representability)] The class \smallmap{S} is representable, in the sense that there is a small map \func{\pi}{E}{U} (a \emph{representation}) of which any other small map \func{f}{Y}{X} is locally (in $X$) a quotient of a pullback. More explicitly: any $\func{f}{Y}{X} \in \smallmap{S}$ fits into a diagram of the form
\diag{Y \ar[d]_f & B \ar[d] \ar[r] \ar@{->>}[l] & E \ar[d]^{\pi} \\
X & A \ar[r] \ar@{->>}[l] & U, }
where the left hand square is covering and the right hand square is a pullback.
\item[(Exactness)] For any equivalence relation \diag{R \ar@{ >->}[r] & X \times X} given by a small mono, a stable quotient $X/R$ exists in \ct{E}.
\end{description}

\bibliographystyle{plain} \bibliography{ast}

\begin{thebibliography}{10}

\bibitem{aczel78}
P.~Aczel.
\newblock The type theoretic interpretation of constructive set theory.
\newblock In {\em Logic Colloquium '77 (Proc. Conf., Wroc\l aw, 1977)},
  volume~96 of {\em Stud. Logic Foundations Math.}, pages 55--66.
  North-Holland, Amsterdam, 1978.

\bibitem{aczelrathjen01}
P.~Aczel and M.~Rathjen.
\newblock Notes on constructive set theory.
\newblock Technical Report No. 40, Institut Mittag-Leffler, 2000/2001.

\bibitem{awodeybirkedalscott02}
S.~Awodey, L.~Birkedal, and D.S. Scott.
\newblock Local realizability toposes and a modal logic for computability.
\newblock {\em Math. Structures Comput. Sci.}, 12(3):319--334, 2002.

\bibitem{berg06}
B.~van~den Berg.
\newblock {\em Predicative topos theory and models for constructive set
  theory}.
\newblock PhD thesis, University of Utrecht, 2006.
\newblock Available from the author's homepage.

\bibitem{bergmoerdijk07b}
B.~van~den Berg and I.~Moerdijk.
\newblock Aspects of predicative algebraic set theory {I}: exact completion.
\newblock Submitted for publication, arXiv:0710.3077, 2007.

\bibitem{bergmoerdijk07a}
B.~van~den Berg and I.~Moerdijk.
\newblock A unified approach to algebraic set theory.
\newblock To be published in the proceedings of the Logic Colloquium 2006,
  arXiv:0710.3066, 2007.

\bibitem{bergmoerdijk07d}
B.~van~den Berg and I.~Moerdijk.
\newblock Aspects of predicative algebraic set theory {III}: sheaf models.
\newblock In preparation, 2008.

\bibitem{birkedal00}
L.~Birkedal.
\newblock {\em Developing theories of types and computability via
  realizability}, volume~34 of {\em Electronic Notes in Theoretical Computer
  Science}.
\newblock Elsevier, Amsterdam, 2000.
\newblock Available at http://www.elsevier.nl/locate/entcs/volume34.html.

\bibitem{birkedalvanoosten02}
L.~Birkedal and J.~van Oosten.
\newblock Relative and modified relative realizability.
\newblock {\em Ann. Pure Appl. Logic}, 118(1-2):115--132, 2002.

\bibitem{carboni95}
A.~Carboni.
\newblock Some free constructions in realizability and proof theory.
\newblock {\em J. Pure Appl. Algebra}, 103:117--148, 1995.

\bibitem{carbonifreydscedrov88}
A.~Carboni, P.J. Freyd, and A.~Scedrov.
\newblock A categorical approach to realizability and polymorphic types.
\newblock In {\em Mathematical foundations of programming language semantics
  (New Orleans, LA, 1987)}, volume 298 of {\em Lecture Notes in Comput. Sci.},
  pages 23--42. Springer, Berlin, 1988.

\bibitem{friedman73}
H.M. Friedman.
\newblock Some applications of {K}leene's methods for intuitionistic systems.
\newblock In {\em Cambridge Summer School in Mathematical Logic (Cambridge,
  1971)}, volume 337 of {\em Lecture Notes in Math.}, pages 113--170. Springer,
  Berlin, 1973.

\bibitem{friedman77}
H.M. Friedman.
\newblock Set theoretic foundations for constructive analysis.
\newblock {\em Ann. of Math. (2)}, 105(1):1--28, 1977.

\bibitem{hofstra06}
P.~J.~W. Hofstra.
\newblock All realizability is relative.
\newblock {\em Math. Proc. Cambridge Philos. Soc.}, 141(2):239--264, 2006.

\bibitem{hofstravanoosten03}
P.J.W. Hofstra and J.~van Oosten.
\newblock Ordered partial combinatory algebras.
\newblock {\em Math. Proc. Cambridge Philos. Soc.}, 134(3):445--463, 2003.

\bibitem{hyland88}
J.~M.~E. Hyland.
\newblock A small complete category.
\newblock {\em Ann. Pure Appl. Logic}, 40(2):135--165, 1988.

\bibitem{hyland82}
J.M.E. Hyland.
\newblock The effective topos.
\newblock In {\em The L.E.J. Brouwer Centenary Symposium (Noordwijkerhout,
  1981)}, volume 110 of {\em Stud. Logic Foundations Math.}, pages 165--216.
  North-Holland, Amsterdam, 1982.

\bibitem{hylandrobinsonrosolini90}
J.M.E. Hyland, E.P. Robinson, and G.~Rosolini.
\newblock The discrete objects in the effective topos.
\newblock {\em Proc. London Math. Soc. (3)}, 60(1):1--36, 1990.

\bibitem{joyalmoerdijk95}
A.~Joyal and I.~Moerdijk.
\newblock {\em Algebraic set theory}, volume 220 of {\em London Mathematical
  Society Lecture Note Series}.
\newblock Cambridge University Press, Cambridge, 1995.

\bibitem{kleenevesley65}
S.C. Kleene and R.E. Vesley.
\newblock {\em The foundations of intuitionistic mathematics, especially in
  relation to recursive functions}.
\newblock North--Holland, Amsterdam, 1965.

\bibitem{kouwenhovenvanoosten05}
C.~Kouwenhoven-Gentil and J.~van Oosten.
\newblock Algebraic set theory and the effective topos.
\newblock {\em J. Symbolic Logic}, 70(3):879--890, 2005.

\bibitem{longley95}
J.~Longley.
\newblock {\em Realizability Toposes and Language Semantics}.
\newblock PhD thesis, Edinburgh University, 1995.

\bibitem{lubarsky06}
R.S. Lubarsky.
\newblock {CZF} and {S}econd {O}rder {A}rithmetic.
\newblock {\em Ann. Pure Appl. Logic}, 141(1-2):29--34, 2006.

\bibitem{mccarty84}
D.C. McCarty.
\newblock {\em Realizability and recursive mathematics}.
\newblock PhD thesis, Oxford University, 1984.

\bibitem{moerdijkpalmgren02}
I.~Moerdijk and E.~Palmgren.
\newblock Type theories, toposes and constructive set theory: predicative
  aspects of {AST}.
\newblock {\em Ann. Pure Appl. Logic}, 114(1-3):155--201, 2002.

\bibitem{vanoosten08}
J.~van Oosten.
\newblock {\em Realizability -- An Introduction to its Categorical Side},
  volume 152 of {\em Studies in Logic}.
\newblock Elsevier, Amsterdam.
\newblock Expected March 2008.

\bibitem{vanoosten94}
J.~van Oosten.
\newblock Axiomatizing higher-order {K}leene realizability.
\newblock {\em Ann. Pure Appl. Logic}, 70(1):87--111, 1994.

\bibitem{vanoosten97}
J.~van Oosten.
\newblock The modified realizability topos.
\newblock {\em J. Pure Appl. Algebra}, 116(1-3):273--289, 1997.

\bibitem{pitts81}
A.M. Pitts.
\newblock {\em The Theory of Triposes}.
\newblock PhD thesis, University of Cambridge, 1981.

\bibitem{rathjen06}
M.~Rathjen.
\newblock Realizability for constructive {Z}ermelo-{F}raenkel set theory.
\newblock In {\em Logic Colloquium '03}, volume~24 of {\em Lect. Notes Log.},
  pages 282--314. Assoc. Symbol. Logic, La Jolla, CA, 2006.

\bibitem{streicher97}
T.~Streicher.
\newblock A topos for computable analysis.
\newblock Unpublished note available from the author's homepage, 1997.

\bibitem{streicher05}
T.~Streicher.
\newblock Realizability models for {CZF}+ $\lnot$ {P}ow.
\newblock Unpublished note available from the author's homepage, March 2005.

\bibitem{troelstra73}
A.~S. Troelstra.
\newblock Notes on intuitionistic second order arithmetic.
\newblock In {\em Cambridge Summer School in Mathematical Logic (Cambridge,
  1971)}, volume 337 of {\em Lecture Notes in Math.}, pages 171--205. Springer,
  Berlin, 1973.

\bibitem{troelstra98}
A.~S. Troelstra.
\newblock Realizability.
\newblock In {\em Handbook of proof theory}, volume 137 of {\em Stud. Logic
  Found. Math.}, pages 407--473. North-Holland, Amsterdam, 1998.

\end{thebibliography}

\end{document}